\documentclass[12pt]{extarticle}
\usepackage[utf8]{inputenc}
\usepackage[T1]{fontenc}

\usepackage{caption}
\oddsidemargin \evensidemargin
\usepackage{amssymb}
\usepackage{amsmath}
\usepackage{comment}
\usepackage{amsthm}
\usepackage{amssymb}
\usepackage{xcolor}
\usepackage{titlesec}
\usepackage{graphicx}
\usepackage{tikz,pgfplots}
\usepackage{comment}
\usepackage{url}
\usepackage{kbordermatrix}
\usepackage[all]{xypic}
\usepackage{tikz-cd} 
\usepackage{algorithm}
\usepackage{algpseudocode}
\usepackage{kbordermatrix}
\usepackage{hyperref}
\usepackage{color}
\hypersetup{
    colorlinks=true,
    linkcolor=black,
    filecolor=magenta,      
    urlcolor=blue,
    pdfpagemode=FullScreen,
    }

\newcommand{\C}{\mathbb{C}}
\newcommand{\Z}{\mathbb{Z}}

\newcommand{\ve}{\varepsilon}
\newcommand{\s}{\sigma}
\renewcommand{\d}{\delta}

\newcommand{\omegah}{\omega_{\d}}
\def\mychange#1{#1}

\newtheorem{theorem}{Theorem}[section]
\newtheorem{lemma}[theorem]{Lemma}

\newtheorem{remark}[theorem]{Remark}
\newtheorem{corollary}[theorem]{Corollary}

\newtheorem{definition}{Definition}
\newtheorem{proposition}[theorem]{Proposition}
\newtheorem{assumption}{Assumption}
\theoremstyle{definition}
\newtheorem{examplex}[theorem]{Example}

\newenvironment{example}
  {\pushQED{\qed}\examplex}
  {\popQED\endexamplex}

\title{Twisted Cohomology and Likelihood Ideals}
\author{Saiei-Jaeyeong Matsubara-Heo and Simon Telen}
\date{}

\begin{document}

\maketitle

\begin{abstract}
\noindent A likelihood function on a smooth very affine variety gives rise to a twisted de Rham complex. We show how its top cohomology vector space degenerates to the coordinate ring of the critical points defined by the likelihood equations. We obtain a basis for cohomology from a basis of this coordinate ring. We investigate the dual picture, where twisted cycles correspond to critical points. We show how to expand a twisted cocycle in terms of a basis, and apply our methods to Feynman integrals from physics. 
\end{abstract}

\section{Introduction}
Very affine varieties are closed subvarieties of an algebraic torus. They have applications in algebraic statistics \cite{huh2014likelihood} and particle physics \cite{mizera2018scattering}. We study smooth such varieties given by hypersurface complements in the algebraic torus. Fix $\ell$ Laurent polynomials $f_1, \ldots, f_\ell \in \mathbb{C}[x_1^{\pm 1}, \ldots, x_n^{\pm 1}]$ in $n$ variables. Localizing at the product $f_1 \cdots f_\ell$ gives the very affine variety
\begin{equation} \label{eq:defX}
 X \, = \, \{ x \in (\mathbb{C}^*)^n \, : \, f_i(x) \neq 0 , \text{ for all } i\} \, = \, (\mathbb{C}^*)^n \setminus V(f_1 \cdots f_\ell).
\end{equation}
This is realized as a closed subvariety of $(\mathbb{C}^*)^{n + \ell}$ via $x \mapsto (x_1,\ldots, x_n,f_1(x)^{-1}, \ldots, f_\ell(x)^{-1})$.
In statistics and physics applications, the functions $f_i$ arise from a \emph{likelihood function} 
\begin{equation} \label{eq:likelihood}
 L(x) \, = \, f^{-s} \, x^\nu \, = \, f_1^{-s_1} \cdots f_\ell^{-s_\ell} x_1^{\nu_1} \cdots x_n^{\nu_n},
 \end{equation}
encountered as the integrand of a \emph{generalized Euler integral} \cite{agostini2022vector,sturmfels2021likelihood}. These are \emph{Bayesian integrals} in statistics, and \emph{Feynman integrals} in physics. 
Outside these applications, generalized Euler integrals are interesting objects in their own right.  
They represent hypergeometric functions and solutions to GKZ systems \cite{GKZ,matsubara2020euler}. Computations with these integrals can be done in a \emph{twisted cohomology} vector space $H^n(X,\omega)$ associated to $X$ and $L$ \cite{agostini2022vector}. 
This paper establishes a crucial relation between $H^n(X,\omega)$ and an ideal in the coordinate ring of $X$, called the \emph{likelihood ideal}.
It makes computations in $H^n(X,\omega)$ explicit, by showing how to compute a basis and how to find coefficients in this basis.

We think of the exponents $s, \nu$ in \eqref{eq:likelihood} as complex parameters, so the likelihood $L$ is multi-valued. The logarithm of $L$ is the \emph{log-likelihood function}, whose partial derivatives are single valued and well-defined on $X$. The complex critical points of the log-likelihood function are the solutions of $\omega(x) = 0$, where $\omega$ is the one-form 
\[ \omega(x) \, = \, {\rm dlog} L(x) \, = \, -s_1 \, {\rm dlog}f_1   - \cdots -s_\ell  \, {\rm dlog}f_\ell  + \frac{\nu_1\, {\rm d}x_1}{x_1} + \cdots +\frac{\nu_n\, {\rm d}x_n}{x_n}.  \]
Expanding $\omega = g_1 {\rm d}x_1 + \cdots + g_n {\rm d}x_n$ in the basis ${\rm d}x_1, \ldots, {\rm d}x_n$ gives $n$ equations $g_1 =  \cdots= g_n=0$ on $X$. The $g_i$ generate an ideal $I$ in the coordinate ring ${\cal O}(X)$ of $X$, called the \emph{likelihood ideal}. June Huh has shown that, for generic $s, \nu$, the likelihood ideal defines $|\chi(X)|$ critical points, with $\chi(X)$ the \emph{Euler characteristic} \cite{huh2013maximum}. This means that $\dim_{\mathbb{C}} {\cal O}(X)/I = |\chi(X)|$.

The Euler characteristic also counts the dimension of the \emph{twisted cohomology} $H^n(X,\omega)$ of $X$ \cite{agostini2022vector}. We briefly recall the definition.
The form $\omega$ is regular on $X$, in the algebraic sense. We write $\omega \in \Omega^1(X)$. More generally, $\Omega^k(X)$ denotes the regular $k$-forms on $X$. Our vector space $H^n(X,\omega)$ is the $n$-th cohomology of the \emph{twisted de Rham complex} $0 \rightarrow \Omega^0(X)\rightarrow \Omega^1(X) \rightarrow \cdots \rightarrow \Omega^n(X) \rightarrow 0$, where the differential is ${\rm d} + \omega \wedge$. In symbols: 
\[ H^n(X,\omega) \, = \, \Omega^n(X)\, / \,  ({\rm d + \omega \wedge})(\Omega^{n-1}(X)).\]
Using the identification $\Omega^n(X) \simeq {\cal O}(X)$ from \eqref{eq:cohomA}, we write this alternatively as $H^n(X,\omega) = {\cal O}(X)/V$, where $V \subset {\cal O}(X)$ is a vector space which is \emph{not} an ideal. We will sometimes write $V(\omega)$ to emphasize the dependence of $V$ on the twist $\omega$. For an element $g \in {\cal O}(X)$, we write $[g]_I$ for its residue class in ${\cal O}(X)/I$, and $[g]_V$ its residue class in ${\cal O}(X)/V$.
In the physics literature, a basis of ${\cal O}(X)/V$ (or, the corresponding set of Feynman integrals) is called a set of \emph{master integrals} \cite{henn2013multiloop}. It is an important computational problem to find such a basis.

For some of our purposes, it will be convenient to keep $s$ and $\nu$ as parameters, rather than fixing complex values. We use the notation $X_K$ for our very affine variety, but now defined over the field $K = \mathbb{C}(s,\nu)$ of rational functions in $s$ and $\nu$. The cohomology module of the twisted de Rham complex is the $K$-vector space ${\cal O}(X_K)/V_K = H^n(X_K,\omega)$. Here $V_K \subset {\cal O}(X_K)$ is the image of ${\rm d} + \omega \wedge$ in the complex $(\Omega^\bullet(X_K),{\rm d} + \omega \wedge)$ over $K$, see Section \ref{sec:2}.

Our first main result establishes an explicit description of $H^n(X_K,\omega)$ as the quotient of a non-commutative ring of difference operators by a left ideal. Let $R$ be the ring of difference operators in $s$ and $\nu$, with coefficients in $K$. Its precise definition is given around Equation \eqref{eq:commutators}.
The following is a simplified version of Theorem \ref{thm:RmodJ}.
\begin{theorem} \label{thm:0intro}
The cohomology $H^n(X_K,\omega) = {\cal O}(X_K)/V_K$ is isomorphic, as a left $R$-module, to the quotient $R/J$, where $J \subset R$ is a left ideal generated by $n + \ell$ difference~operators.
\end{theorem}

The generators of $J$ are given explicitly in Theorem \ref{thm:RmodJ}. Loosely speaking, $J$ is a non-commutative variant of the likelihood ideal mentioned above. That ideal is here reinterpreted as an ideal $I_K$ in ${\cal O}(X_K)$. The similarity between $J$ and $I_K$ gives some intuition behind the next result, which says that bases of $H^n(X_K,\omega)$ can be found from bases of ${\cal O}(X_K)/I_K$.

\begin{theorem} \label{thm:2intro}
The $K$-vector spaces ${\cal O}(X_K)/V_K$ and ${\cal O}(X_K)/I_K$ have dimension $|\chi(X)|$. 
If $\{\beta_1,\dots,\beta_\chi\}\subset{\cal O}(X_K)$ represents a constant basis of ${\cal O}(X_K)/I_K$, in the sense of Definition \ref{def:constant}, then $\{[\beta_1]_{V_K}, \ldots, [\beta_\chi]_{V_K}\}$ is a $K$-basis of ${\cal O}(X_K)/V_K$.
\end{theorem}

Theorem \ref{thm:2intro} has an analog over $\C$, though one needs to be careful when specializing $s$ and $\nu$.
Computing a basis for ${\cal O}(X)/I$ can be done by computing the critical points numerically (Algorithm \ref{alg:basis}). This is a task of \emph{numerical nonlinear algebra} \cite[Section 5]{agostini2022vector}. 

\begin{theorem} \label{thm:1intro}
Let $(s, \nu) \in \mathbb{C}^{\ell + n}$ be generic complex parameters in the sense of Assumption \ref{assum:lefschetz} and let $\{[\beta_1]_I, \ldots, [\beta_\chi]_I\}$ be a basis for ${\cal O}(X)/I$. 
The set $\{[\beta_1]_{V(\omega/\d)}, \ldots, [\beta_\chi]_{V(\omega/\d)}\}$ is a basis for $H^n(X,\omega/\d) = {\cal O}(X)/V(\omega/\d)$, for almost all $\d \in \C \setminus \{0\}$.
\end{theorem}

Under stronger assumptions (Remark \ref{rem:delta1}) one can use $\d = 1$ in this theorem. However, there are special choices of $s, \nu$ and $[\beta_i]_I$ for which this does not work, see Example \ref{ex:special}. 

The connection between the twisted cohomology and the likelihood ideal is made explicit by a \emph{degeneration}. This is formalized in Section \ref{sec:3}. We introduce a new parameter $\d$, so that making $\d$ move from $1$ to $0$ turns ${\cal O}(X_K)/{V_K}$ into ${\cal O}(X_K)/I_K$.
This degeneration also appears in \cite{matsubara2022localization}, where it was used to relate the cohomology intersection pairing to Grothendieck's residue pairing. 
It can be defined over $\mathbb{C}$ as well, and turns Theorems \ref{thm:2intro} and \ref{thm:1intro} into practice: bases of cohomology turn into bases of the likelihood quotient.

Section \ref{sec:4} relates our results to some standard bilinear pairings from the literature, namely the \emph{cohomology intersection pairing} and the \emph{period pairing}. For instance, the former is characterized as a unique bilinear pairing compatible with the $R$-module structure from Theorem \ref{thm:0intro}. This is Theorem \ref{thm:uniqueness}.
The period pairing uses the \emph{twisted homology} $H_n(X,-\omega) = {\rm Hom}_{\mathbb{C}}({\cal O}(X)/V, \mathbb{C})$ \cite[Section 2]{agostini2022vector}.
A preferred basis of $H_n(X,-\omega)$ in this article consists of the \emph{Lefschetz thimbles} \cite[Section 3]{witten2011analytic}. These are called \emph{Lagrangian cycles} in \cite[Section 4.3]{aomoto2011theory}. There is one Lefschetz thimble $\Gamma_j \subset X$ for each critical point $x^{(j)}$ satisfying $\omega(x^{(j)}) = 0$, with the property that $x^{(j)} \in \Gamma_j$. They represent the linear functionals 
\[ [g]_V \, \longmapsto \, \int_{\Gamma_j} g(x) \cdot L(x) \,  \frac{{\rm d} x_1}{x_1} \wedge \cdots \wedge  \frac{{\rm d} x_n}{x_n} . \]
In Section \ref{sec:4}, we will show that when $\d \rightarrow 0$, these Lefschetz thimbles degenerate to the \emph{evaluation functionals} $[g]_I \longmapsto g(x^{(j)})$ on the likelihood quotient ${\cal O}(X)/I$.



Next to finding bases of cohomology (Theorems \ref{thm:2intro} and \ref{thm:1intro}), we also address the following problem. Given a basis $[\beta_1]_{V_K}, \ldots, [\beta_\chi]_{V_K}$ of ${\cal O}(X_K)/{V_K}$ and an element $[g] \in {\cal O}(X_K)/{V_K}$, find the coefficients $c_i \in K$ of $g$ in this basis: $[g] = c_1 \, [\beta_1]_{V_K} + \cdots + c_\chi \, [\beta_\chi]_{V_K}$. 
The unknown coefficients $c_i \in K$ are found from a set of \emph{contiguity matrices} for the ideal $J$ from Theorem \ref{thm:0intro}. 
These are $\chi \times \chi$ matrices over $K$ which encode how the difference operators act on the basis elements $[\beta_i]_{V_K}$.
Contiguity matrices for $\d \rightarrow 0$ become pairwise commuting $K$-linear maps representing \emph{multiplication modulo} $I_K$ (Theorem \ref{thm:multiplication}). 
We show how to compute these matrices and provide an implementation. Our algorithm, inspired by \emph{border basis} algorithms \cite{mourrain1999new}, exploits the fact that a basis for ${\cal O}(X_K)/V_K$ can be computed a priori. For the physics application, it offers an alternative for \emph{Laporta's algorithm} to systematically compute all \emph{integration by parts} relations among Feynman integrals \cite{laporta2000high}.

The paper is organized as follows. Section \ref{sec:2} recalls the twisted de Rham complex and establishes Theorem \ref{thm:0intro}. Section \ref{sec:3} introduces the likelihood ideal and sets up the degeneration which takes the twisted de Rham cohomology to the likelihood quotient. It contains a proof of Theorem \ref{thm:2intro}. In Section \ref{sec:4}, we prove Theorem \ref{thm:1intro} via our degeneration and a perfect pairing of cohomology. We also discuss a different perfect pairing with twisted homology, whose degeneration turns Lefschetz thimbles into evaluation at critical points. Section \ref{sec:5} deals with our computational goals: computing bases for cohomology and computing expansions in this basis. We implement our algorithms in \texttt{Julia}. The code uses the packages \texttt{HomotopyContinuation.jl} \cite{breiding2018homotopycontinuation} and \texttt{Oscar.jl} \cite{OSCAR}. It is made available at \url{https://mathrepo.mis.mpg.de/TwistedCohomology}. In Section \ref{sec:6}, we apply our methods to compute contiguity matrices in several examples, including some Feynman integrals. 

\section{Twisted de Rham cohomology}\label{sec:2}
Fix $\ell$ Laurent polynomials $f_1, \ldots, f_\ell \in \mathbb{C}[x_1^{\pm 1}, \ldots, x_n^{\pm 1}]$ and let $X_{\mathbb{C}} = X$ be as in \eqref{eq:defX}. As alluded to in the introduction, we will need analogous schemes $X_A$ over different rings $A$. In our setting, the ring $A$ satisfies $A = \mathbb{C}$ or $\mathbb{C}[s_1,\ldots, s_\ell,\nu_1, \ldots, \nu_n] \subset A$. Our schemes $X_A$ are
\[ X_A \, = \, {\rm Spec} \, A[x_1^{\pm 1}, \ldots, x_n^{\pm 1}]_{f_1 \cdots f_\ell}. \]
The $A$-module of \emph{regular $k$-forms} on $X_A$ is denoted by 
\begin{equation} \label{eq:regforms}
\Omega^k(X_A) \, = \, \left \{ \sum_{1 \leq j_1< \ldots < j_k \leq n} g_{j_1,\ldots,j_k} \, {\rm d} x_{j_1} \wedge \cdots \wedge {\rm d}x_{j_k} \, \big | \,  g_{j_1, \ldots,j_k} \in \sum_{(a,b)\in \mathbb{Z}^{\ell + n}} A \cdot f^a \, x^b \right \} .
\end{equation}
To construct the algebraic twisted de Rham complex, we consider the one-form
\begin{align} \label{eq:omega}
\begin{split}
\omega \, &= \, - s_1 \, {\rm dlog} f_1 - \cdots  - s_\ell \, {\rm dlog} f_\ell + \frac{\nu_1\, {\rm d}x_1}{x_1} + \cdots +\frac{\nu_n\, {\rm d}x_n}{x_n} \\
 &= \, \sum_{j=1}^n \left( \frac{\nu_j}{x_j} -  s_1 \frac{\frac{\partial f_1}{\partial x_j}}{f_1} - \cdots -s_\ell \frac{\frac{\partial f_\ell}{\partial x_j}}{f_\ell} \right ) {\rm d} x_j \quad \in \Omega^1(X_A). 
 \end{split}
\end{align}
When $A = \mathbb{C}$, $s$ and $\nu$ in this formula are generic tuples of complex numbers. The precise meaning of \emph{generic} is given in Definition \ref{def:generic}.  When $\mathbb{C}[s,\nu] \subset A$, the coefficients of $\omega$ are variables. The one-form $\omega$ is the logarithmic differential of our likelihood function \eqref{eq:likelihood}.
The \emph{twisted differential} $\nabla_\omega: \Omega^k(X_A) \rightarrow \Omega^{k+1}(X_A)$ is given by $\nabla_\omega(\phi) = ({\rm d} + \omega \wedge)\, \phi$.
Here ${\rm d}$ acts on $g \in A[x,x^{-1}]$ by exterior derivation in $x$.
This gives a cochain complex
\begin{equation} \label{eq:cochain}
(\Omega^\bullet(X_A), \nabla_\omega) : \, 0 \longrightarrow \Omega^0(X_A)\overset{\nabla_\omega}{\longrightarrow} \Omega^1(X_A) \overset{\nabla_\omega}{\longrightarrow} \cdots \overset{\nabla_\omega}{\longrightarrow} \Omega^n(X_A) \longrightarrow 0.
\end{equation}
The cohomology of this complex $H^k(\Omega^\bullet(X_A),\nabla_\omega)$ is denoted by $H^k(X_A,\omega)$ for simplicity. We are mainly interested in the top cohomology.
This is the $A$-module
\begin{equation} \label{eq:cohomA}
H^n(X_A,\omega) \, = \,  \frac{\Omega^n(X_A)}{\nabla_\omega(\Omega^{n-1}(X_A))} \, = \, {\cal O}(X_A)/V_A.
\end{equation}
Here ${\cal O}(X_A) = A[x_1^{\pm 1}, \ldots, x_n^{\pm 1}]_{f_1 \cdots f_\ell}$ is the ring of regular functions on $X_A$. The submodule $V_A$ is the image  of $\nabla_\omega(\Omega^{n-1}(X_A))$ under the identification $\Omega^n(X_A) \simeq {\cal O}(X_A)$ which takes $1 \in {\cal O}(X_A)$ to the canonical volume form $
\frac{{\rm d}x}{x} =\frac{{\rm d}x_1}{x_1}\wedge\cdots\wedge\frac{{\rm d}x_n}{x_n} 
$ on the $n$-dimensional algebraic torus. Concretely, ${\cal O}(X_A) \ni g \sim g \frac{{\rm d} x}{x} \in  \Omega^n(X_A)$ and $ {\cal O}(X_A)/V_A \ni [g]_{V_A}  \sim [g \frac{{\rm d} x}{x}] \in H^n(X_A,\omega)$.

\begin{example} \label{ex:expand}
    It is instructive to explicitly write down some elements of $\nabla_\omega(\Omega^{n-1}(X_A))$. An $(n-1)$-form in $\Omega^{n-1}(X_A)$ is an $A$-linear combination of elements of the form 
    \begin{equation} \label{eq:n-1form}
    (-1)^{j-1}\,f^a x^b \, {\rm d} x_{\widehat{j}} \, = \,(-1)^{j-1}\, f^a x^b \, {\rm d} x_1 \wedge \cdots \wedge {\rm d} x_{j-1} \wedge {\rm d} x_{j+1} \wedge \cdots \wedge {\rm d} x_n. 
    \end{equation}
    The image of \eqref{eq:n-1form} under the twisted differential $\nabla_\omega$ is 
    \[ \nabla_\omega((-1)^{j-1}\,f^a x^b \, {\rm d} x_{\widehat{j}}) \, = \, f^a x^b \left ( \frac{\nu_j + b_j}{x_j} - \sum_{i=1}^\ell (s_i - a_i) \frac{\frac{\partial{f_i}}{\partial x_j}}{f_i}  \right) {\rm d}x_1 \wedge \cdots \wedge {\rm d}x_n. \qedhere \] 
\end{example}
If $A$ is a field, e.g.~$A = \mathbb{C}$ or $A = K = \mathbb{C}(s,\nu)$, then $H^n(X_A,\omega)$ in an $A$-vector space. Theorem \ref{thm:dimcohom} below shows that its dimension depends only on the topology of $X_{\mathbb{C}}$. 

Before stating this dimension result, we clarify the meaning of \emph{generic} $s, \nu$. Consider a smooth projective compactification $\bar{X}$ of $X = X_{\mathbb{C}}$ such that the boundary $D=\bar{X}\setminus X$ is a simple normal crossing divisor, with irreducible decomposition $D=\bigcup_{k=1}^{\ell + n} D_k$.
From this compactification, we define a $\mathbb{Z}$-linear form ${\rm Res}_{D_k}(s,\nu)$ for each divisor $D_k$. The coefficients are the orders of vanishing of $f_i^{-1}, x_j$ along $D_k$, see Example \ref{ex:special} and \cite[Lemma A.2]{agostini2022vector}.
\begin{definition} \label{def:generic}
    The parameters $(s,\nu)\in\mathbb{C}^{\ell+n}$ are \emph{generic} if  ${\rm Res}_{D_k}(s,\nu)\notin \mathbb{Z}$ for any $k$.
\end{definition}

\noindent
This notion of genericity is not a Zariski open condition in $\C^{\ell + n}$.
It is, however, a mild requirement: generic $(s,\nu)$ are open and dense in the standard topology of $\mathbb{C}^{\ell+n}$.
\begin{theorem} \label{thm:dimcohom}
    If $A = \mathbb{C}$ and $s, \nu$ are generic in the sense of Definition \ref{def:generic}, or $A = K = \mathbb{C}(s,\nu)$, then ${\rm dim}_A \, H^n(X_A,\omega) = (-1)^n \cdot \chi(X_{\mathbb{C}})$, where $\chi(\cdot)$ denotes the Euler characteristic.
\end{theorem}
The proof of Theorem \ref{thm:dimcohom} will make use of a technical lemma. Let $\bar{X}_K$ be a smooth projective compactification of $X_K$ obtained from $\bar{X} = \bar{X}_\C$ above via the base extension $\mathbb{C} \rightarrow K$.
For $A = \C$ or $K$, let $\Omega^p_{A,\log}$ denote the sheaf of $p$-forms on $\bar{X}_A$, logarithmic along $D$. Moreover, we write $\Omega^p_{A,\log}(kD) = \Omega^p_{A,\log}\otimes_A\mathcal{O}_{\bar{X}_A}(kD)$ for any integer $k$. 
\begin{lemma} \label{lem:specseq}
    If $A = \mathbb{C}$ and $s, \nu$ are generic in the sense of Definition \ref{def:generic}, or $A = K = \mathbb{C}(s,\nu)$, then the Hodge-to-de Rham spectral sequence 
   \begin{equation} \label{eq:specseq}
    E^{p,q}_1 \, = \, \mathbb{H}^q(\bar{X}_A \, ; \, \Omega^p_{A,\log}(kD))\,\, \Rightarrow \, \, \mathbb{H}^{p+q}(\bar{X}_A \, ; \, (\Omega_{A,\log}^\bullet(kD),\nabla_\omega))
\end{equation}
    degenerates at the $E_1$ stage for sufficiently large $k$, and $\dim_A H^n(X_A,\omega)$ equals
    \begin{equation} \label{eq:specseq2}
    \dim_A \mathbb{H}^{n}(\bar{X}_A;(\Omega_{A,\log}^\bullet(kD),\nabla_\omega)) \, = \, \sum_{p+q=n}\dim_A\mathbb{H}^q(\bar{X}_A;\Omega^p_{A,\log}(kD)).
\end{equation}
\end{lemma}
\begin{proof}
    The degeneration of \eqref{eq:specseq} is a consequence of the vanishing theorem by Grothendieck and Serre, see \cite[Proposition 2.6.1]{EGA}.
    The second claim follows from that degeneration.
\end{proof}

\begin{proof}[Proof of Theorem \ref{thm:dimcohom}]
We set $\Omega_{A,\log}^p(*D):=\Omega_{A,\log}^p\otimes_A\mathcal{O}_{\bar{X}_A}(*D)$, with $\mathcal{O}_{\bar{X}_A}(*D)$ the sheaf of rational functions on $\bar{X}_A$ with poles along $D$.
    Our cohomology vector space equals the hypercohomology group
    \begin{equation} \label{eq:hypercohom}
    H^n(X_A,\omega) \, = \, \mathbb{H}^{n}(\bar{X}_A;(\Omega_{A}^\bullet(*D),\nabla_\omega)). 
    \end{equation}
    The canonical morphism $(\Omega_{A,\log}^\bullet(kD),\nabla_\omega)\rightarrow(\Omega_{A}^\bullet(*D),\nabla_\omega)$ is a quasi-isomorphism for any $k\in\mathbb{Z}$.
    This follows from the same argument as \cite[Properties 2.9]{esnault1992lectures}.
    Therefore, \eqref{eq:hypercohom} implies $H^n(X_A,\omega)=\mathbb{H}^{n}(\bar{X}_A;(\Omega_{A,\log}^\bullet(kD),\nabla_\omega))$. Applying \eqref{eq:specseq2}, we have for sufficiently large $k$ that
\begin{equation*}
   \dim_A H^n(X_A,\omega) \,  = \, \sum_{p+q=n}\dim_A\mathbb{H}^q(\bar{X}_A;\Omega^p_{A,\log}(kD)).
\end{equation*}
Here, to apply Lemma \ref{lem:specseq} for $A = \C$ we need the genericity assumption (Definition \ref{def:generic}). If $A = \C$, the statement is proved, as the righthand side equals $(-1)^n \cdot \chi(X_{\mathbb{C}})$. 
For $A = K$, note that $\dim_K\mathbb{H}^q(\bar{X}_K; \Omega^p_{K,\log}(kD))=\dim_{\C}\mathbb{H}^q(\bar{X}_\C;\Omega^p_{\C,\log}(kD))$, as there is a canonical isomorphism $\mathbb{H}^q(\bar{X}_K;\Omega^p_{K,\log}(kD))\simeq \mathbb{H}^q(\bar{X}_{\C};\Omega^p_{\C,\log}(kD))\otimes_{\C}K$. This gives
\[
    \sum_{p+q=n}\dim_K\mathbb{H}^q(\bar{X}_K;\Omega^p_{K,\log}(kD))  \, =  \, \sum_{p+q=n}\dim_{\C}\mathbb{H}^q(\bar{X}_\C;\Omega^p_{\C,\log}(kD)).
\]
We conclude $\dim_K H^n(X_K,\omega) = (-1)^n \cdot \chi(X_{\mathbb{C}})$, and we are done.
\end{proof}

For applications in physics, where our likelihood function is a Feynman integrand in Lee-Pomeransky representation \cite{lee2013critical}, the relevant case is $\ell = 1$. A basis for $H^n(X_A,\omega)$ corresponds to a set of \emph{master integrals} \cite{henn2013multiloop}. Relations between Feynman integrals come from $A$-linear relations modulo $V_A$. Computing such bases and relations is the topic of Section \ref{sec:5}. Our algorithms rest on the main result of this section, which is Theorem \ref{thm:RmodJ} below.

We set $A = K = \mathbb{C}(s,\nu)$.
We introduce a non-commutative ring of \emph{difference operators} $R=K\langle \s_{s_1}^{\pm 1}, \ldots, 
\s_{s_\ell}^{\pm 1},\s_{\nu_1}^{\pm 1},\dots,\s_{\nu_n}^{\pm 1}\rangle$, generated by $\s_s, \s_{\nu_j}$ and their inverses, with relations 
\begin{equation} \label{eq:commutators}
  [\s_{s_i}, \s_{s_j}] = [\s_{s_i},\s_{\nu_j}]=[\s_{\nu_i},\s_{\nu_j}]=0, \, \,  \ [\s_{s_i},s_j]= \delta_{ij} \s_{s_j}, \, \, [\s_{\nu_i},\nu_j]=\delta_{ij}\s_{\nu_i}.
\end{equation}
Here $[ a, b ]=ab-ba$ is the commutator in the ring $R$ and $\delta_{ij}$ is Kronecker's delta. Let $e_j$ be the $j$-th standard basis vector. The difference operators $\sigma_{s_i}$ and $\sigma_{\nu_j}$ act on ${\cal O}(X_K)$ by
\begin{equation} \label{eq:Raction}
  \s_{s_i} \bullet g(s,\nu) =f_i^{-1} \cdot g(s+e_i,\nu) ,\quad \s_{\nu_j} \bullet g(s,\nu) = x_j \cdot g(s,\nu+e_j).
\end{equation}
The notation $g(s,\nu)$ emphasizes the dependence of the regular function $g \in {\cal O}(X_K)$ on $s$ and $\nu$. The action of $R$ is obtained by extending \eqref{eq:Raction} $K$-linearly. 
This action turns ${\cal O}(X_K)$ into a left $R$-module. 
Moreover, using notation from Example \ref{ex:expand}, the observation that $\sigma_{s_i} \bullet \nabla_\omega(f^a x^b {\rm d}x_{\widehat{j}}) = \nabla_\omega(f^{a-e_i}x^b {\rm d}x_{\widehat{j}})$ and $\sigma_{\nu_k} \bullet \nabla_\omega(f^a x^b {\rm d}x_{\widehat{j}}) = \nabla_\omega(f^{a}x^{b+e_k} {\rm d}x_{\widehat{j}} )$ shows that the $R$-action is well defined modulo $V_K$, so that ${\cal O}(X_K)/V_K$ is a left $R$-module as well.
Our proof of the next theorem relies on $D$-module theory, in particular on results from \cite{loeser1991equations}. It briefly recalls the main relevant concepts. The reader is referred to \cite{loeser1991equations,SST,sattelberger2019d} for details.
\begin{theorem} \label{thm:RmodJ}
    As a left $R$-module, the cohomology $H^n(X_K,\omega) = {\cal O}(X_K)/V_K$ is isomorphic to the quotient $R/J$ of $R$ by the left ideal $J \subset R$ generated by
  \begin{equation}\label{eqn:2}
    1-\s_{s_i} f_i(\s_\nu), \text{ for }i = 1, \ldots, \ell, \quad \text{ and } \quad \s_{\nu_j}^{-1} \nu_j - \sum_{i=1}^\ell s_i \, \s_{s_i} \frac{\partial f_i}{\partial x_j}(\s_{\nu}), \text{ for } j = 1, \ldots, n.
  \end{equation}
  Moreover, the isomorphism sends the residue class of $\s_s^a \s_{\nu}^b$ to $[f^{-a}x^b]_{V_K} \in {\cal O}(X_K)/V_K$.
\end{theorem}
\begin{proof}
   Let $T^{\ell+n}_A = {\rm Spec} \, A[x_1^{\pm 1}, \ldots, x_n^{\pm 1}, z_1^{\pm 1}, \ldots, z_\ell^{\pm 1}]$ be the $(\ell+n)$-dimensional algebraic torus over a field $A$.
   Its Weyl algebra $D_{\ell+n,A} = D_{T^{\ell+n}_{A}}$ consists of linear differential operators in the $\ell + n$ variables $x, z$ with coefficients in $A[x^{\pm 1}, z^{\pm 1}]$.
   Our very affine variety $X$ is naturally embedded into $T^{\ell + n}_{\mathbb{C}}$ via $x \mapsto(x_1, \ldots, x_n, f_1(x)^{-1}, \ldots, f_\ell(x)^{-1})$.
   The \emph{local cohomology} of $X \subset T^{\ell + n}_{\mathbb{C}}$ is the $D_{\ell+n,\mathbb{C}}$-module $M = D_{\ell+n,\mathbb{C}}/H$, where $H$ is the left $D_{\ell+n,\mathbb{C}}$-ideal generated~by 
    \begin{equation} \label{eq:localcohom} 1-z_if_i, \text{ for }i = 1, \ldots, \ell, \quad \text{ and } \quad - x_j\partial_{x_j} + \sum_{i=1}^\ell z_i \partial_{z_i} z_i  x_j\frac{\partial f_i}{\partial x_j}, \text{ for } j = 1, \ldots, n. 
    \end{equation}
    We consider two different ways to construct a left $R$-module from $M$:
    \begin{enumerate}
        \item[(1)] 
        {In \cite[Th\'eor\`eme 1.2.1, Lemme 1.2.2]{loeser1991equations}, ${\frak M} ( \cdot )(s,\nu)$ is a functor from the category of left $D_{\ell + n,\mathbb{C}}$-modules to the category of left $R$-modules. It performs the algebraic Mellin transform ${\frak M}(\cdot)$ of a left $D_{\ell + n,\mathbb{C}}$-module and then applies the tensor product with $K = \mathbb{C}(s,\nu)$. The first step replaces the action of $-x_j\partial_{x_j}$ with that of $\nu_j$, $-z_i \partial_{z_i}$ with $s_i$, $x_j$ with $\sigma_{\nu_j}$, and $z_i$ with $\sigma_{s_i}$. We claim that ${\frak M}(M)(s,\nu) = {\frak M}(M) \otimes_{\mathbb{C}[s,\nu]} K = R/J$.
        }
        \item[(2)] We construct the left $D_{\ell+n,K}$-module $M(s,\nu) \bullet z^{s}x^\nu$ by applying linear differential operators in $D_{\ell+n,K}$ to $z^{s} x^\nu$, and regarding the result modulo $H_K \bullet z^{s}x^\nu$, where $H_K$ is the left $D_{\ell+n,K}$-ideal generated by \eqref{eq:localcohom}.
        In symbols, $M(s,\nu) = M \otimes_{\mathbb{C}[s,\nu]} K$.
        We consider the push-forward ${\cal H}\{M \} = \pi_+(M(s,\nu) \bullet z^{s}x^\nu)$ in the sense of $D$-modules under the constant map $\pi: T^{\ell + n}_K \rightarrow {\rm Spec} \, K$.
        We claim that ${\cal H}\{M\} = H^n(X_K,\omega) = {\cal O}(X_K)/V_K$.
    \end{enumerate}
     The theorem follows from these claims, as ${\frak M}(M)(s,\nu) \simeq {\cal H}\{M\}$ by a result of Loeser and Sabbah \cite[Lemme 1.2.2]{loeser1991equations}. Claim (1) is easily verified by observing that ${\frak M}(D_{\ell+n,\mathbb{C}})(s,\nu)= R$ and the Mellin transform turns the generators in \eqref{eq:localcohom} into \eqref{eqn:2}.
     For claim (2), let $Z \subset T_K^{n+\ell}$ be the natural embedding of $X_K$ in $T_K^{n+\ell}$. Its equations are $ 1-z_1f_1(x)= \cdots  = 1-z_\ell f_\ell(x) = 0$. We write $Z \overset{\iota}{\hookrightarrow} T_K^{n+\ell}$ for the inclusion.
In view of Kashiwara's equivalence \cite[Chapter VI, Theorem 7.13]{Borel} and following the notation of \cite{Borel}, we obtain a sequence of isomorphisms
 \begin{align}
     \mathcal{H} \{ M \}
    \,  &= \, \pi_+\left(\mathbb{R}\Gamma_Z\mathcal{O}_{T_K^{n+\ell}}[\ell]\otimes\mathcal{O}z^sx^{\nu}\right) \, = \, \pi_+\left(\iota_+\iota^!\mathcal{O}_{T_K^{n+\ell}}[\ell]\otimes\mathcal{O}z^sx^{\nu}\right) \label{eq:8}\\                   &\simeq \, \pi_+\left(\iota_+\mathcal{O}_{Z}z^sx^{\nu}\right)  \simeq \, \pi_+\left(\iota_+\mathcal{O}_{Z}f^{-s}x^{\nu}\right) \simeq \, (\pi\circ\iota)_+\mathcal{O}_{Z}f^{-s}x^{\nu} \label{eq:9}\\ \vspace{0.2cm}
     &= \, H^n(X_K,\omega)\nonumber.
 \end{align}
Here passing from \eqref{eq:8} to \eqref{eq:9} uses $\iota^!\mathcal{O}_{T_K^{n+\ell}}[\ell]\simeq \mathcal{O}_Z$. 
\end{proof}

\begin{remark}
The authors of \cite{BBKP} exploit the difference module structure only in the $\nu$-variables, for $\ell=1$.
In fact, the parametric annihilator ideal in that paper arises from the Mellin transform of $H$ in the $z$-direction, viewed as a module over the Weyl algebra $D_{n,K}$.
One recovers $H^n(X_K,\omega)$ by applying the Mellin transform in the $x$-direction.

\end{remark}

\begin{example}[$n = 2, \ell = 3$] \label{ex:M05}
Consider the very affine surface $X = (\mathbb{C}^*)^2 \setminus V(f_1f_2f_3)$, where $f_1 = x-1, f_2 = y-1, f_3 = x-y$. This variety can be identified with the moduli space ${\cal M}_{0,5}$ of five points on $\mathbb{P}^1$ \cite[Section 2]{sturmfels2021likelihood}. Its real part is the complement of an arrangement of five lines in $\mathbb{R}^2$. By Varchenko's theorem \cite[Proposition 1]{sturmfels2021likelihood}, the Euler characteristic equals the number of bounded polygons in that complement, which is two. The generators of $J$ are
\begin{align} \label{eq:genM05}
1 - \sigma_{s_1}&(\sigma_{\nu_1} - 1), \quad 1 - \sigma_{s_2}(\sigma_{\nu_2} - 1), \quad  1 - \sigma_{s_3}(\sigma_{\nu_1} - \sigma_{\nu_2}), \\  &\sigma_{\nu_1}^{-1} \nu_1 - s_1 \sigma_{s_1} - s_3 \sigma_{s_3}, \quad \, \sigma_{\nu_2}^{-1} \nu_2 - s_2 \sigma_{s_2} + s_3 \sigma_{s_3}. \qedhere
\end{align}
\end{example}

Theorem \ref{thm:RmodJ} reduces computations in the cohomology $H^n(X_K,\omega) = {\cal O}(X_K)/V_K$ to computations in the difference ring $R$ modulo the left ideal $J$. 
Our algorithm in Section \ref{sec:5} is inspired by a generalization of commutative {Gr\"obner bases}, called \emph{border bases} \cite{mourrain1999new}. It makes use of the fact that a $K$-basis of ${\cal O}(X_K)/V_K$ is known a priori via Theorems \ref{thm:1intro} and \ref{thm:2intro}.

\section{Degeneration and likelihood ideals} \label{sec:3}

Theorem \ref{thm:RmodJ} expresses our cohomology vector space $H^n(X_K,\omega)$ as a quotient of a \emph{non-commutative} ring $R$ by a left ideal $J$. In this section, we introduce a degeneration which turns the cohomology into a quotient of the \emph{commutative} ring ${\cal O}(X_K)$ by the \emph{likelihood ideal}. For the moment, we will switch back to our general setting where $X_A$ is defined over a ring $A$, which is either $\C$ or contains $\C[s,\nu]$. Our degeneration is interesting for at least two reasons. First, it preserves bases in the sense of Theorem \ref{thm:2intro}, which allows us to compute a basis for $H^n(X_A,\omega)$ from a basis of the likelihood quotient. This section features a proof of Theorem \ref{thm:2intro}. Second, the degeneration provides new insights into the relation between critical points and twisted homology. This will be explored in Section \ref{sec:4}. We remark that our degeneration is much like a \emph{Gr\"obner deformation}, in the sense of \cite{SST}, which turns $R$ into its associated graded ring ${\cal O}(X_K)$. This was pointed out to us by Bernd Sturmfels.

The term \emph{likelihood} comes from maximum likelihood estimation, where one seeks to maximize the log-likelihood function $\log L(x)$, with $L(x)$ as in \eqref{eq:likelihood}.
The \emph{likelihood equations} are obtained by equating its partial derivatives with respect to $x_1, \ldots, x_n$ to zero.
This leads to $\omega = 0$, where $\omega \in \Omega^1(X_A)$ is as in \eqref{eq:omega}.
The critical points form a zero-dimensional subscheme of $X_A$, defined by an ideal $I_A\subset{\cal O}(X_A)$ called the \emph{likelihood ideal}:
\begin{equation} \label{eq:IA}
    I_A \, = \, \left \langle \frac{\nu_j}{x_j} -  s_1 \frac{\frac{\partial f_1}{\partial x_j}}{f_1} - \cdots -s_\ell \frac{\frac{\partial f_\ell}{\partial x_j}}{f_\ell}, \, \, j = 1, \ldots, n \right  \rangle \, \subset \, {\cal O}(X_A). 
\end{equation} 
Note that, for $A = K$, the similarity of these generators with \eqref{eqn:2} hints at a strong connection between ${\cal O}(X_K)/I_K$ and $H^n(X_K,\omega) = {\cal O}(X_K)/V_K$. This section explores that connection. 

Identifying ${\cal O}(X_A)$ with regular $n$-forms $\Omega^n(X_A)$, we observe that $I_A$ is the image of the map $\Omega^{n-1}(X_A) \rightarrow \Omega^n(X_A)$ given by $\phi \mapsto \omega \wedge \phi$. Together with \eqref{eq:cohomA}, this gives
\begin{equation} \label{eq:degen} 
H^n(X_A,\omega) \, = \, \frac{\Omega^n(X_A)}{({\rm d}+ \omega \wedge)\Omega^{n-1}(X_A)}, \quad {\cal O}(X_A)/I_A \, = \, \frac{\Omega^n(X_A)}{\omega \wedge \Omega^{n-1}(X_A) }.
\end{equation}
These equations are the ingredients to explain our intuition behind this section's degeneration. 
To go from left to right in \eqref{eq:degen}, it suffices to drop the `${\rm d}$' in the twisted differential. This motivates us to introduce a parameter $\d$ into the twisted de Rham complex \eqref{eq:cochain} as follows. We replace the twisted differential $\nabla_\omega = {\rm d} + \omega \wedge$ by $\nabla_\omega^\d = \d {\rm d} + \omega \wedge$. When $\d = 1$, we recover our original complex \eqref{eq:cochain}. When $\d = 0$, we obtain a complex of ${\cal O}(X_A)$-modules, which is the dual Koszul complex of the likelihood ideal \eqref{eq:IA} (more precisely, of its $n$ generators $g_1, \ldots, g_n$). Since $g_1, \ldots, g_n \in {\cal O}(X_A)$ form a regular sequence, this dual Koszul complex is a free resolution. This implies that all its cohomology modules are zero, except at level $n$, where it equals our likelihood quotient ${\cal O}(X_A)/I_A$. Below, we will make this more precise. 
\begin{remark}
    The deformation parameter $\delta$ corresponds to the reciprocal of the $\epsilon$-parameter in \emph{dimensional regularization} from physics (substitute $d=d_0-2\epsilon$ in formula (2.5) of \cite{lee2013critical}).
\end{remark}

To formally introduce the degeneration parameter $\d$ into our cochain complex, we add it to our field $K = \C(s,\nu)$. Since we want to analyze what happens near $\d = 0$, we choose to work over the power series ring $K[\![\d]\!]$. To simplify the notation, we will write $X_{\d} = X_{K[\![\d]\!]}$.
The \emph{$\d$-twisted differential} $\nabla_\omega^\d: \Omega^k(X_\d) \rightarrow \Omega^{k+1}(X_\d)$ is given by $\nabla_\omega^\d(\phi) = (\d{\rm d} + \omega \wedge)\, \phi$. Here regular $k$-forms $\Omega^k(X_\d)$ are defined as in \eqref{eq:regforms} with $A = K[\![\d]\!]$.
The $k$-th cohomology group of 
\begin{equation} \label{eq:degencomplex}
(\Omega^\bullet(X_\d), \nabla_\omega^\d) : \, 0 \longrightarrow \Omega^0(X_\d)\overset{\nabla_\omega^\d}{\longrightarrow} \Omega^1(X_\d) \overset{\nabla_\omega^\d}{\longrightarrow} \cdots \overset{\nabla_\omega^\d}{\longrightarrow} \Omega^n(X_\d) \longrightarrow 0
\end{equation}
is denoted by $H^k(X_\d,\omegah)$. Tensoring \eqref{eq:degencomplex} with the Laurent series $K(\!(\d)\!)$ we obtain the cochain complex $(\Omega^\bullet(X_{K(\!(\d)\!)}), \nabla_\omega^\d)$, with cohomology $H^k(X_{K(\!(\d)\!)},\omega_\d)$.  We will see below (Corollary 
\ref{cor:laurentseries}) that the dimension for $k = n$ is the signed Euler characteristic, which is reminiscent of Theorem \ref{thm:dimcohom}. For $A = K[\![\d]\!]$, the analogous statement is the following. 
\begin{theorem} \label{thm:rankdegen}
    The cohomology $H^n(X_\d, \omegah)$ is a free $K[\![\d]\!]$-module of rank $(-1)^n \cdot \chi(X_\C)$.
\end{theorem}
Before proving Theorem \ref{thm:rankdegen}, it is convenient to formalize the notion of `driving $\d$ to 0'.
\begin{definition}
    For a $K[\![\d]\!]$-module $M$, we define the $K$-vector space $\lim_{\d \rightarrow 0} M = M/(\d \cdot M)$.
\end{definition}
\begin{example} \label{ex:limit}
    One checks that $\lim_{\d \rightarrow 0} \Omega^k(X_\d) = \Omega^k(X_K)$.
\end{example}
This justifies our claim that cohomology degenerates to the likelihood quotient: 
\begin{lemma} \label{lem:limithomol}
    The limit $\lim_{\d \rightarrow 0} H^n(X_\d, \omegah)$ is ${\cal O}(X_K)/I_K$ and has dimension $(-1)^n \cdot \chi(X_\C)$.
\end{lemma}
\begin{proof}
    We first observe that our limit equals 
    \[ \lim_{\d \to 0} H^n(X_\d, \omegah) \, = \, \frac{H^n(X_\d, \omegah)}{\d \cdot H^n(X_\d,\omegah)} \, = \, \frac{\Omega^n(X_\d)}{\nabla_\omega^\d(\Omega^{n-1}(X_\d)) + \d \cdot \Omega^n(X_\d)}. \]
    Using $\Omega^k(X_\d) = K[\![\d]\!]\otimes_K \Omega^k(X_K)$, we see that the denominator on the right equals
    \[  (\d {\rm d} + \omega \wedge) \left (K[\![\d]\!]\otimes_K\Omega^{n-1}(X_K) \right ) + \d \cdot \Omega^n(X_\d) \, = \, \omega \wedge \Omega^{n-1}(X_K) + \d \cdot \Omega^n(X_\d).\]
    Together with Example \ref{ex:limit} this leads to $\lim_{\d \to 0}H^n(X_\d, \omegah) = \frac{\Omega^n(X_K)}{\omega \wedge \Omega^{n-1}(X_K)} = {\cal O}(X_K)/I_K$. The statement about the dimension follows from \cite[Theorem 1]{huh2013maximum}.
\end{proof}
Our proof of Theorem \ref{thm:rankdegen} will also use the following two lemmas.
\begin{lemma} \label{lem:commalg}
    Let $M$ be a finitely generated $K[\![\d]\!]$-module, with free part of rank $r$. We have
    \begin{enumerate}
        \item $\dim_{K(\!(\d)\!)} K(\!(\d)\!) \otimes_{K[\![\d]\!]} M = r$ and 
        \item $\dim_K \lim_{\d \to 0} M \geq r$, where equality holds if and only if $M$ is free.
    \end{enumerate}
\end{lemma}
\begin{proof}
    The proof requires only elementary commutative algebra. We present a sketch and leave details to the reader. For the first statement, the torsion part of $M$ is annihilated by the tensor product, since $K(\!(\d)\!) \otimes_{K[\![\d]\!]} K[\![\d]\!]/\langle \d^p \rangle = 0$ for any $p$. For the second statement, note that a torsion component $K[\![\d]\!]/\langle \d^p \rangle$ has nonzero contribution to the dimension $\dim_K \lim_{\d \to 0} M = \dim_K M/(\d \cdot M)$. Indeed, we have $\lim_{\d \to 0} K[\![\d]\!]/\langle \d^p \rangle = K$.
\end{proof}

The final lemma is an analog of Lemma \ref{lem:specseq} over $K[\![\d]\!]$ and $K(\!(\d)\!)$.
\begin{lemma}\label{lem:specseq d}
    If $A = K[\![\d]\!]$  or $A = K(\!(\d)\!)$, then the Hodge-to-de Rham spectral sequence 
   \begin{equation} \label{eq:specseq e}
    E^{p,q}_1 \, = \, \mathbb{H}^q(\bar{X}_A \, ; \, \Omega^p_{A,\log}(kD))\,\, \Rightarrow \, \, \mathbb{H}^{p+q}(\bar{X}_A \, ; \, (\Omega_{A,\log}^\bullet(kD),\nabla_\omega^\d))
\end{equation}
    degenerates at the $E_1$ stage for sufficiently large $k$.
    In particular, $\dim_{K(\!(\d)\!)} H^n(X_{K(\!(\d)\!)},\omega_\d)$ is
    \begin{equation}
    \dim_{K(\!(\d)\!)} \mathbb{H}^{n}(\bar{X}_{K(\!(\d)\!)};(\Omega_{K(\!(\d)\!),\log}^\bullet(kD),\nabla_\omega^\d)) \, = \, \sum_{p+q=n}\dim_{K(\!(\d)\!)}\mathbb{H}^q(\bar{X}_{K(\!(\d)\!)};\Omega^p_{K(\!(\d)\!),\log}(kD)).
\end{equation}
\end{lemma}
\noindent
The proof of this lemma is the same as that of Lemma \ref{lem:specseq}. The following is a consequence. 
\begin{corollary} \label{cor:laurentseries}
    The dimension of the $K(\!(\d)\!)$-vector space $H^n(X_{K(\!(\d)\!)},\omegah)$ is $(-1)^n \cdot \chi(X_\C)$.
\end{corollary}
\begin{proof}
The proof is identical to that of Theorem \ref{thm:dimcohom}, replacing Lemma \ref{lem:specseq} by Lemma \ref{lem:specseq d}.
\end{proof}

\begin{proof}[Proof of Theorem \ref{thm:rankdegen}]
    We first show that $H^n(X_\d, \omegah)$ is finitely generated over $K[\![\d]\!]$.
    Because the residue of the $\d$-connection $(\mathcal{O}_{X_\d}(kD),\nabla_\omega^\d)$ along a component $D_j$ is ${\rm Res}_{D_j}(s,\nu)+k\d$, the natural morphism $(\Omega_{\d,\log}^\bullet(kD),\nabla_\omega^\d)\to (\Omega_{\d,\log}^\bullet(*D),\nabla_\omega^\d)$ is a quasi-isomorphism by the argument in the proof of \cite[Properties 2.9]{esnault1992lectures}.
    It follows that
    \[ H^n(X_\d,\omega_\d) \, \simeq \, \mathbb{H}^n(\bar{X}_\d \, ; \, (\Omega_{\d,\log}^\bullet(*D),\nabla_\omega^\d)) \, \simeq \, \mathbb{H}^{n}(\bar{X}_\d \, ; \, (\Omega_{\d,\log}^\bullet(kD),\nabla_\omega^\d)). \]
    By Lemma \ref{lem:specseq d}, this implies that $H^n(X_\d, \omegah)$ is finitely generated over $K[\![\d]\!]$.
    By Lemma \ref{lem:limithomol}, 
    \begin{equation} \label{eq:arg1}
    \dim_K \lim_{\d \to 0} H^n(X_\d, \omegah) = (-1)^n \cdot \chi(X_\C).
    \end{equation}
    Tensoring \eqref{eq:degencomplex} with $K(\!(\d)\!)$ we obtain \begin{equation} \label{eq:arg2}
    \dim_{K(\!(\d)\!)} K(\!(\d)\!) \otimes_{K[\![\d]\!]} H^n(X_\d,\omegah) \, = \,  \dim_{K(\!(\d)\!)} H^n(\Omega^\bullet(X_{K(\!(\d)\!)}), \nabla_\omega^\d) \, = \, (-1)^n \cdot \chi(X_\C),
    \end{equation}
    where the second equality is Corollary \ref{cor:laurentseries}. Lemma \ref{lem:commalg}, \eqref{eq:arg1} and \eqref{eq:arg2} show that $H^n(X_\d, \omegah)$ is free of rank $(-1)^n \cdot \chi(X_\C)$.
\end{proof}

Our next goal is to prove Theorem \ref{thm:2intro}. We first need to define \emph{constant bases}. Recall that $[g]_{I_K}$ denotes the residue class of $g \in {\cal O}(X_K)$ in the likelihood quotient ${\cal O}(X_K)/I_K$. 
\begin{definition} \label{def:constant}
A subset $\{\beta_1,\dots,\beta_\chi\}\subset \mathcal{O}(X_K)$ is said to represent a constant basis for $\mathcal{O}(X_K)/I_K$ if $\beta_1,\dots,\beta_\chi\in\mathcal{O}(X_\C)$ and $\{[\beta_1]_{I_K},\dots,[\beta_\chi]_{I_K}\}$ forms a $K$-basis of $\mathcal{O}(X_K)/I_K$.
\end{definition}
Our proof of Theorem \ref{thm:2intro} uses the notation $V_{A} = \nabla_\omega^\d(\Omega^{n-1}(X_A))$ for $A = K[\![\d]\!]$ or $K(\!(\d)\!)$.

\begin{proof}[Proof of Theorem \ref{thm:2intro}]
By assumption, $\{\beta_1,\dots,\beta_\chi\} \subset {\cal O}(X)={\cal O}(X_\C)$ and $\chi =|\chi(X)|$.
Consider the $K[\![\d]\!]$-submodule $N \subset H^n(X_\d,\omega_\d)$ generated~by $\{ [\beta_1]_{V_{K[\![\d]\!]}},\dots,[\beta_\chi]_{V_{K[\![\d]\!]}}\}$.
Here $[\beta_i]_{V_{K[\![\d]\!]}}$ is the residue class of $\beta_i \in {\cal O}(X) \subset {\cal O}(X_\d)$ in $H^n(X_\d, \omegah)$.
By our assumption, we have $\lim_{\d\to 0} H^n(X_\d,\omega_\d)/N=0$.
Nakayama's lemma implies $H^n(X_\d,\omega_\d)=N$.
By Equation \eqref{eq:arg2}, $H^n(X_{K(\!(\d)\!)} , \omegah)$ has dimension $\chi$, and therefore $\{ [\beta_1]_{V_{K(\!(\d)\!)}},\dots,[\beta_\chi]_{V_{K(\!(\d)\!)}}\}$ is a basis.
Next, {we consider a field extension $\iota:K\hookrightarrow K(\!(\d)\!)$ given by 
\begin{equation}\label{eq:field_extension}
K\ni a(s,\nu)\overset{\iota}{\mapsto} a(s/\delta,\nu/\delta)\in K(\!(\d)\!)
\end{equation}
and} define a morphism of $K(\!(\d)\!)$-vector spaces $\varphi: K(\!(\d)\!) \otimes_{K} H^n(X_K,\omega) \rightarrow H^n(X_{K(\!(\d)\!)},\omegah)$: 
\begin{equation}\label{eq:varphi}
     h \otimes [g(s,\nu)]_{V_K} \, \overset{\varphi}{\longmapsto} \, \left [  h \cdot  g \left ( \frac{s}{\d}, \frac{\nu}{\d} \right ) \right ]_{V_{K(\!(\d)\!)}}. 
\end{equation}
Note that $\varphi$ is well-defined: if $g(s,\nu) = \nabla_\omega(\phi(s,\nu))$ with $\phi \in \Omega^{n-1}(X_K)$, we have 
\[ h\cdot g\left( \frac{s}{\d}, \frac{\nu}{\d} \right ) \, = \, h \cdot \left ({\rm d} + \omega\left( \frac{s}{\d}, \frac{\nu}{\d} \right ) \wedge  \right) \phi \left ( \frac{s}{\d}, \frac{\nu}{\d} \right ) \, = \, \nabla_\omega^\d  \left ( \frac{h}{\d} \cdot  \phi \left ( \frac{s}{\d}, \frac{\nu}{\d} \right ) \right )  \, \in \, V_{K(\!(\d)\!)}. \]
Since $\varphi$ is a surjective map of equidimensional $K(\!(\d)\!)$-vector spaces (Theorems \ref{thm:dimcohom} and \ref{thm:rankdegen}), it is an isomorphism. 
Because $\beta_i \in {\cal O}(X)$, we have $[\beta_i]_{V_{K(\!(\d)\!)}} = \varphi(1 \otimes [\beta_i]_{V_K})$. It follows easily that $[\beta_1]_{V_K}, \ldots, [\beta_\chi]_{V_K}$ is a basis for $H^n(X_K,\omega)$. The theorem is proved. 
\end{proof}

The proof of Theorem \ref{thm:2intro} immediately implies the following corollary.
\begin{corollary}\label{cor:free_basis}
If $\{\beta_1,\dots,\beta_\chi\}\subset{\cal O}(X_K)$ represents a constant basis of ${\cal O}(X_K)/I_K$, then $\{[\beta_1]_{V_{K[\![\d]\!]}},\dots,[\beta_\chi]_{V_{K[\![\d]\!]}}\}$ is a free basis of $H^n(X_\d,\omega_\d)$.  
\end{corollary}

Finally, we link the degeneration back to Theorem \ref{thm:RmodJ} by bringing $\d$ into our ring of difference operators $R$.
We set $R_\d = K[\![\d]\!] \langle \s_{s_1}^{\pm 1}, \ldots, 
s_{s_\ell}^{\pm 1},\s_{\nu_1}^{\pm 1},\dots,\s_{\nu_n}^{\pm 1}\rangle$. Now $\d$ commutes with any element of $R_\d$, and the remaining commutator rules are as follows: 
\begin{equation} \label{eq:commutatorshbar}
  [\s_{s_i}, \s_{s_j}] = [\s_{s_i},\s_{\nu_j}]=[\s_{\nu_i},\s_{\nu_j}]=0, \, \,  \ [\s_{s_i},s_j]= \delta_{ij} \d \s_{s_j}, \, \, [\s_{\nu_i},\nu_j]=\delta_{ij} \d \s_{\nu_i}.
\end{equation}
Note that when $\d = 1$, we recover the relations \eqref{eq:commutators}. On the other hand, when $\d = 0$, $R_0$ is the (commutative!) coordinate ring of the $(n + \ell)$-dimensional algebraic torus $T_K^{n + \ell} = (K \setminus \{0\})^{n + \ell}$.
The action of $R_\d$ on $g(s,\nu) \in \Omega^n(X_\d)$, generalizing \eqref{eq:Raction}, is
\begin{equation} \label{eq:Ractionhbar}
  \s_{s_i} \bullet g(s,\nu) =f_i^{-1} \cdot g(s+ \d \cdot e_i,\nu) ,\quad \s_{\nu_j} \bullet g(s,\nu) = x_j \cdot g(s,\nu+\d \cdot e_j).
\end{equation}
This makes $\Omega^n(X_\d)$ a left $R_\d$-module.
Also, \eqref{eq:Ractionhbar} is well-defined on cohomology, as
\[ \sigma_{s_i} \bullet \nabla_\omega^\d(f^a x^b {\rm d}x_{\widehat{j}}) = \nabla_\omega^\d(f^{a-e_i}x^b {\rm d}x_{\widehat{j}})\quad \text{and} \quad \sigma_{\nu_k} \bullet \nabla_\omega^\d(f^a x^b {\rm d}x_{\widehat{j}}) = \nabla_\omega^\d(f^{a}x^{b+e_k} {\rm d}x_{\widehat{j}} ). \]
Hence, also $H^n(X_\d,\omegah)$ is a left $R_\d$-module. Here is a version of Theorem \ref{thm:RmodJ} in this setting. 
\begin{theorem} \label{thm:RmodJhbar}
    As a left $R_\d$-module, the cohomology $H^n(X_\d,\omega_\d)$ is isomorphic to the quotient $R_\d/J_\d$ of $R_\d$ by the left ideal $J_\d \subset R_\d$ generated by \eqref{eqn:2}.
  Moreover, the isomorphism sends the residue class of $\s_s^a \s_{\nu}^b$ to $[f^{-a}x^b\frac{dx}{x}] \in H^n(X_\d, \omega_\d)$. In particular, $R_\d/J_\d \simeq H^n(X_\d,\omegah)$ as $K[\![\d]\!]$-modules and $\lim_{\d \to 0} R_\d/J_\d = {\cal O}(X_K)/I_K$.
\end{theorem}

\begin{proof}
We set $M_\d=R_\d/J_\d$.
Since $\bigcap_{n=0}^\infty \d^nM_\d=0$ and $\dim_K \lim_{\d\to0} M_\d=\dim_K {\cal O}(X_K)/I_K=\chi<\infty$, $M_\d$ is finitely generated by \cite[Theorem 8.4]{Matsumura}.
On the other hand, one can prove that the naturally induced morphism $\psi:K(\!(\d)\!)\otimes_{K[\![\d]\!]}M_\d\to H^n(X_{K(\!(\d)\!)},\omega_\d)$ is an isomorphism.
This is because $\psi$ is the base extension of the isomorphism $R/J\to H^n(X_K;\omega)$ of Theorem \ref{thm:RmodJ} via the field extension $\iota:K\hookrightarrow K(\!(\d)\!)$ given by \eqref{eq:field_extension}.
It follows that $M_\d$ is a free $K[\![\d]\!]$-module by Lemma \ref{lem:commalg}.
The theorem follows from the obvious fact that the morphism $M_\d\to H^n(X_\d;\omega_\d)$ is surjective.
\end{proof}

We end the section by illustrating some of its results in a one-dimensional example.
\begin{example}\label{exa:cubic}
We take $n=\ell=1$ and $f(x)=1-x^3$.
As a basis of the likelihood quotient, we take $\{ [1]_{I_K},[x]_{I_K},[x^2]_{I_K}\}\subset\mathcal{O}(X_K)/I_K$.
By Theorem \ref{thm:2intro}, $\{[\frac{{\rm d}x}{x}],[{{\rm d}x}],[x{{\rm d}x}]\}$ is a $K$-basis of $H^1(X_K,\omega)$.
Through the isomorphism $H^1(X_K,\omega)\simeq R/J$ of Theorem \ref{thm:RmodJ}, it corresponds to a set $\{ [1],[\s_\nu],[\s_\nu^2]\}\subset R/J$.
On the other hand, Corollary \ref{cor:free_basis} implies that the set $\{ [1],[\s_\nu],[\s_\nu^2]\}$ is a free basis of $R_\delta/J_\delta\simeq H^1(X_\d,\omega_\d)$.
The representation matrix of the $K[\![\d]\!]$-linear map $\s_\nu:H^1(X_\d,\omega_\d)\to H^1(X_\d,\omega_\d)$ is given by 
\begin{equation}\label{eq:multiplication}
    \left(
    \begin{array}{ccc}
0& 0& \frac{ {\nu}}{   {\nu}-3s+ 3\d} \\
 1& 0& 0 \\
0&  1& 0 \\
\end{array}\right).
\end{equation}
We will show how to compute such matrices in general in Section \ref{sec:5} (Algorithm \ref{alg:better_alg}). 
Taking the limit $\d\to 0$ in \eqref{eq:multiplication}, the matrix \eqref{eq:multiplication} converges to the representation matrix of the multiplication map $\mathcal{O}(X_K)/I_K\to \mathcal{O}(X_K)/I_K$ which sends $[g]_{I_K}$ to $[x\cdot g]_{I_K}$, with respect to the basis $\{ [1]_{I_K},[x]_{I_K},[x^2]_{I_K}\}$.
The eigenvalues of this matrix for $\d \to 0$ are solutions to the likelihood equation $\omega=0$, see Theorem \ref{thm:multiplication}.
\end{example}

\section{Perfect pairings}
\label{sec:4}

A main goal in this paper is to compare $H^n(X,\omega)$ to the likelihood quotient ${\cal O}(X)/I$. This section illustrates how the degeneration in Section \ref{sec:3} turns certain bilinear pairings between $H^n(X,\omega)$ and its dual into pairings between ${\cal O}(X)/I$ and its dual. On the non-commutative side, i.e.~the side of $H^n(X,\omega)$, we will discuss \emph{period pairings} and \emph{intersection pairings}. When $\delta \rightarrow 0$, these turn into \emph{evaluation pairings} and \emph{Grothendieck residue pairings} respectively. 

\subsection{Period pairing}
Dual to the twisted cohomology $H^n(X,\omega) = H^n(X_\C,\omega)$ is the \emph{twisted homology} $H_n(X,-\omega)$.
Its elements $[\Gamma] \in H_n(X,-\omega)$ are \emph{twisted cycles}. A representative $\Gamma$ of $[\Gamma]$ is a singular cycle, together with a choice of a branch of the likelihood function $L(x)$ on $\Gamma$.
The minus sign in the notation $H_n(X,-\omega)$ is justified by the observation that these branches are local solutions $\phi$ to the differential equation $\nabla_{-\omega}(\phi) = ({\rm d} - \omega \wedge) \phi = 0$. See \cite[Chapter 2]{aomoto2011theory}, or \cite[Section 2]{agostini2022vector} for more details.
The \emph{period pairing} between cohomology and homology leads to the \emph{generalized Euler integrals} mentioned in the introduction, including marginal likelihood integrals and Feynman integrals. We denote this period pairing by $\langle \cdot, \cdot \rangle_{\rm per}^\omega$, and use the short notation $\langle g^+, \Gamma^- \rangle_{\rm per}^\omega = \langle [g^+]_V, [\Gamma^{-}] \rangle_{\rm per}^\omega$. The signs in this notation record the fact that cohomology is defined with the twist `$+\omega\wedge$', and homology with `$-\omega \wedge$'. The definition of the period pairing is $\langle \cdot, \cdot \rangle_{\rm per}^\omega: H^n(X,\omega) \times H_n(X,-\omega) \rightarrow \C$, with
\begin{equation} \label{eq:per}
\langle g^+, \Gamma^- \rangle_{\rm per}^\omega \, = \, \int_{\Gamma^-} g^+(x) \cdot L(x) \, \frac{{\rm d} x_1}{x_1} \wedge \cdots \wedge  \frac{{\rm d} x_n}{x_n}.  
\end{equation}
Here the likelihood function $L(x) = f^{-s}x^\nu$ is as in \eqref{eq:likelihood}, and the dependence on $\omega$ is through $L = {\rm exp}(\log L) = {\rm exp}(\int \omega)$. Note that $\langle \cdot, \cdot \rangle_{\rm per}^\omega$ is a $\C$-bilinear map on $H^n(X,\omega) \times H_n(X,-\omega)$. 
The period pairing is \emph{perfect}, meaning that it identifies $H^n(X,\omega)$ as the vector space dual of $H_n(X,-\omega)$, and vice versa. Throughout the section, when we work over $\C$, we assume genericity of $s, \nu$ as in Definition \ref{def:generic}, and we set $\chi =  (-1)^n \cdot \chi(X)$. When a basis $[\beta_1^+]_V,\ldots, [\beta_\chi^+]_V$ for $H^n(X,\omega)$ and a basis $[\Gamma_1^-], \ldots, [\Gamma_\chi^-]$ for $H_n(X,-\omega)$ are fixed, the period pairing is represented by a square matrix $P(\omega)$ of size $\chi \times \chi$, see Theorem \ref{thm:dimcohom}. Its entries~are 
\[ P(\omega)_{ij} \, = \,  \langle b_i^+,\Gamma_j^- \rangle_{\rm per}^\omega. \]
Flipping the sign of $\omega$ in cohomology and homology, we also have a period pairing $\langle \cdot, \cdot \rangle_{\rm per}^{-\omega}: H^n(X,-\omega) \times H_n(X,\omega) \rightarrow \C$, so that the formula for $\langle g^-, \Gamma^+ \rangle_{\rm per}^{-\omega}$ is similar to \eqref{eq:per}: 
\begin{equation} \label{eq:-per}
\langle g^-, \Gamma^{+} \rangle_{\rm per}^{-\omega} \, = \, \int_{\Gamma^{+} } g^{-}(x) \cdot L(x)^{-1} \, \frac{{\rm d} x_1}{x_1} \wedge \cdots \wedge  \frac{{\rm d} x_n}{x_n}.
\end{equation}
Here $[g^-]_{V(-\omega)} \in H^n(X,-\omega)$, where $V(-\omega) \simeq \nabla_{-\omega}(\Omega^{n-1}(X))$ is defined for $-\omega$ as $V$ was defined for $\omega$, $\Gamma^+ \in H_n(X,\omega)$ and $L(x)$ is as in \eqref{eq:likelihood}.
We fix a basis $[\beta_1^-]_{V(-\omega)},\ldots, [\beta_\chi^-]_{V(-\omega)}$ for $H^n(X,-\omega)$ and a basis $[\Gamma_1^+], \ldots, [\Gamma_\chi^+]$ for $H_n(X,\omega)$ to obtain a matrix $P(-\omega)_{ij} = \langle b_i^-, \Gamma_j^+\rangle_{\rm per}^{-\omega}$.


\subsection{Lefschetz thimbles and intersection pairing}
We will now fix basis cycles $\Gamma_j^-$ and $\Gamma_j^+$, using a construction from Morse theory explained in \cite{aomoto2011theory,witten2011analytic}. These cycles are called \emph{Lefschetz thimbles} or \emph{Lagrangian cycles}. This works nicely under some mild assumptions on $L(x)$, which we will now specify. Let \[ {\rm Hess}(x) \, = \, \det \left ( \frac{\partial^2 \log L(x) }{\partial x_i \partial x_j} \right) \] be the Hessian determinant of $\log L(x)$. Equivalently, ${\rm Hess}(x)$ is the Jacobian determinant
of the $n$ generators in \eqref{eq:IA}. Let $\log L(x) = G(x) + \sqrt{-1} \cdot H(x)$, with $G$ and $H$ real valued. 
\begin{assumption} \label{assum:lefschetz}
    The parameters $s, \nu$ are generic in the sense of Definition \ref{def:generic}. The log-likelihood function $\log L(x)$ has $\chi$ critical points $x^{(1)}, \ldots, x^{(\chi)}$, and $h_j = {\rm Hess}(x^{(j)}) \neq 0$ for $j = 1, \ldots, \chi$. Moreover, the values $H(x^{(j)}), j = 1, \ldots, \chi$ are distinct, and so are the $G(x^{(j)})$.
\end{assumption}
This gives a notion of \emph{genericity} for $s,\nu$ which is slightly stronger than Definition \ref{def:generic}. We will make Assumption \ref{assum:lefschetz} throughout the rest of the section.
The construction of the Lefschetz thimbles is classical, but quite technical. It is explained at length in \cite[Section 4.3]{aomoto2011theory}. The summary is as follows. 
By the Cauchy-Riemann equations, the real critical points of $(u,v) \mapsto G(u + \sqrt{-1}\cdot v)$ coincide with the complex critical points $x = u + \sqrt{-1}\cdot v$ of $\log L(x)$. We denote these critical points by $x^{(1)}, \ldots, x^{(\chi)}$. The function $G$ defines a vector field on $X$ which is a slight modification of its gradient field \cite[\S 4.3.4]{aomoto2011theory}. Along trajectories $x(t)$ of this field, $G$ increases and $H$ is~constant:
\begin{equation} \label{eq:decrease}
\frac{{\rm d} G(x(t))}{{\rm d} t} \, > \, 0 , \quad \frac{{\rm d} H(x(t))}{{\rm d} t} \, = \, 0.
\end{equation}
For each critical point $x^{(j)}$, the Lefschetz thimbles $\Gamma_j^-$ and $\Gamma_j^+$ are unions of~trajectories:
\[ \Gamma_j^{\mp} \, = \, \left \{ x_0 \in X \, : \, \begin{matrix} \text{the trajectory $x(t)$ with $x(0) = x_0$} \text{ satisfies $\lim_{t \rightarrow \pm \infty} x(t) = x^{(j)}$}
\end{matrix} 
\right \}.\]
Both $\Gamma_j^-$ and $\Gamma_j^+$ are $n$-dimensional real manifolds containing $x^{(j)}$. By \eqref{eq:decrease} and Assumption \ref{assum:lefschetz}, they do not contain any of the other critical points $x^{(i)}, i \neq j$. When restricted to $\Gamma_j^-$ ($\Gamma_j^+$), $G$ reaches a maximum (minimum) at $x^{(j)}$, and tends to $-\infty$ ($+\infty$) away from $x^{(j)}$. Intuitively, this explains why the integrals \eqref{eq:per} and \eqref{eq:-per} converge when $\Gamma^\mp = \Gamma_j^{\mp}$. 

\begin{example}[$n=1$] \label{ex:lefschetz}
Figure \ref{fig:lef} visualizes (part of) the Lefschetz thimbles for the data
\begin{equation} \label{eq:exlef}
n = \ell = 1,\qquad f = 1-x^3, \qquad s= 1/2 + 3 \sqrt{-1}, \qquad \nu = 1/7 + 7 \sqrt{-1}.
\end{equation}
\begin{figure}
    \centering
    \includegraphics[width = 7cm]{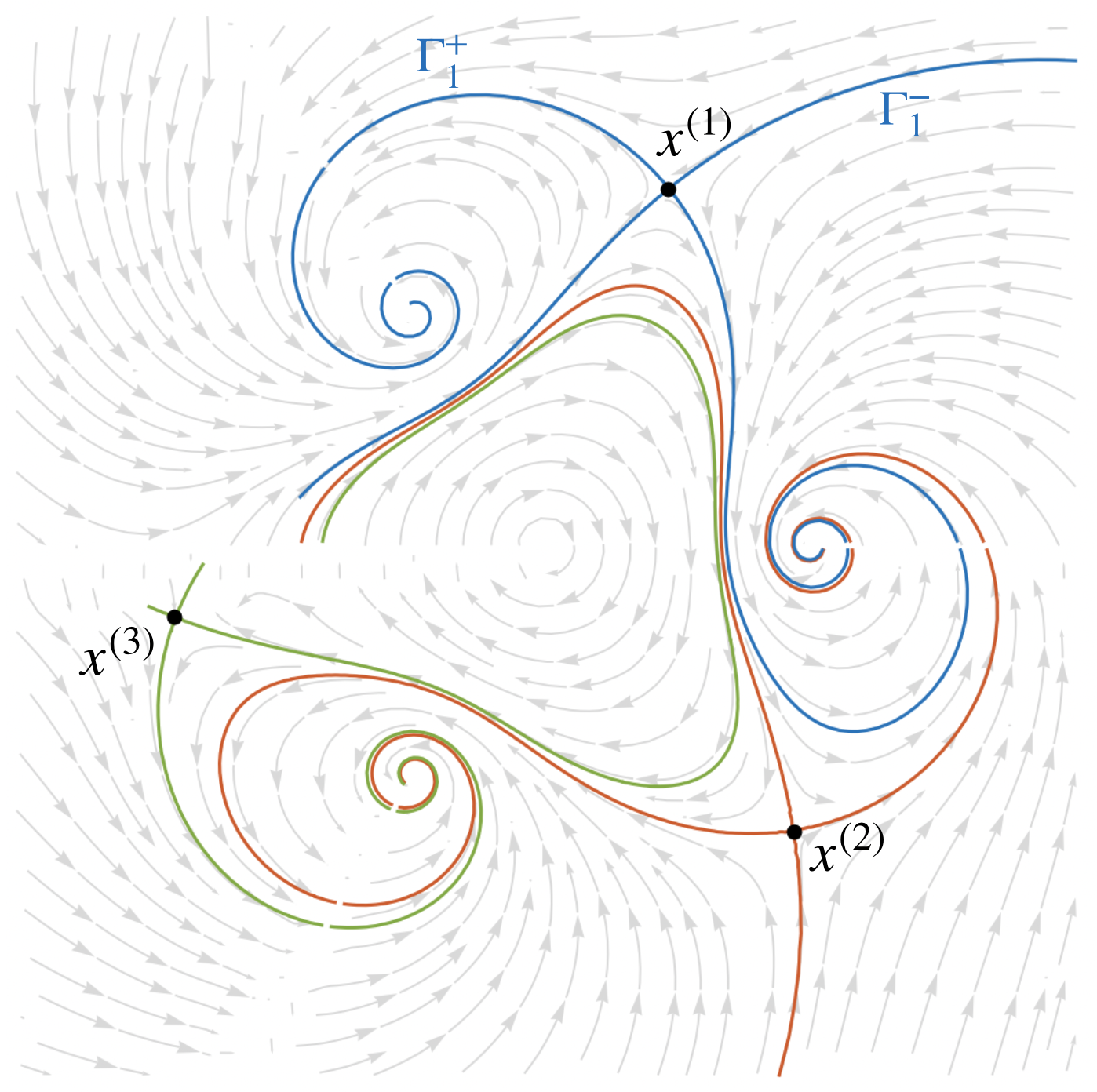}
    \caption{Lefschetz thimbles for the data from \eqref{eq:exlef}.}
    \label{fig:lef}
\end{figure}
The \texttt{Mathematica} code used to generate this picture is available at  \url{https://mathrepo.mis.mpg.de/TwistedCohomology}. The Euler characteristic of $X$ is $\chi(X) = -\chi = - 3$. The three critical points of $\log L(x)$ are the black dots in the picture. The thimbles $\Gamma_j^+$ and $\Gamma_j^-$ are shown in the same color (blue, orange or green), for $j = 1, 2, 3$. They are flow lines of a vector field that is a scaled version of the gradient field of $G(x) = {\rm Re}(\log f^{-s}x^\nu)$ \cite[\S 4.3.4]{aomoto2011theory}, visualized in the background of the figure. The points flowing \emph{towards}/\emph{away from} the critical point are $\Gamma_j^-/\Gamma_j^+$.
\mychange{The thimbles connect the stationary points $x^{(i)}$ of the gradient field with some of the points $\{ 0,1,\frac{-1\pm\sqrt{-3}}{2},\infty\}$.}
We plotted this using \texttt{ContourPlot} on the imaginary part $H(x)$, see \eqref{eq:decrease}. This is tricky because $H$ is multi-valued. Figure \ref{fig:lef} shows the level lines $\{H(x) = H(x^{(j)}) + k \cdot \pi\}$, for a few integer values of $k$. 
\end{example}

To turn our Lefschetz thimbles into twisted cycles, we need to choose which branch of our multi-valued likelihood function $L$ to integrate in \eqref{eq:per} and \eqref{eq:-per}. We do this by selecting a value $L(x_0)^{\pm 1}$ at some $x_0 \in \Gamma_j^\pm$, and analytically continuing along $x(t)$ for $t \in \mathbb{R}$.

\begin{lemma} \label{lem:lefschetzbasis}
    Under Assumption \ref{assum:lefschetz}, the Lefschetz thimbles $[\Gamma_j^{\pm}]$ form a basis for $H_n(X,\pm \omega)$. 
\end{lemma}
\begin{proof}
    Under Assumption \ref{assum:lefschetz}, we have $\dim_\C H_n(X,\pm \omega) = \chi$. By the discussion preceding \cite[Theorem 4.7]{aomoto2011theory}, the Lefschetz thimbles generate $H_n(X,\pm \omega)$. This implies the Lemma.  
\end{proof}

One of the advantages of using Lefschetz thimbles as a basis for twisted homology is that it gives an easy formula for the \emph{intersection pairing} between the cohomology spaces $H^n(X,\omega)$ and $H^n(X,-\omega)$. This is a bilinear map $\langle \cdot, \cdot \rangle_{\rm ch}^\omega: H^n(X,\omega) \times H^n(X,-\omega) \rightarrow \C$ with 
\begin{equation} \label{eq:intersectionpairing}
\langle g^+, g^- \rangle_{\rm ch}^\omega \, = \,  \sum_{j=1}^\chi \,  \langle g^+ , \Gamma_j^- \rangle_{\rm per}^\omega \cdot \langle g^- , \Gamma_j^+ \rangle_{\rm per}^{-\omega}.
\end{equation}
This formula is a special instance of \cite[Equation (2.1)]{matsubara2022localization} and \cite[Equation (18)]{mizera2018scattering}, which holds when Lefschetz thimbles are used as bases for homology. Just like the period pairing, the intersection pairing is perfect. It has gained recent interest in physics, as it turns out to compute \emph{scattering amplitudes} in some special cases \cite{mizera2018scattering}. It follows from the formula \eqref{eq:intersectionpairing} that the matrix $Q(\omega)$ of $\langle \cdot, \cdot \rangle_{\rm ch}^\omega$, using the (arbitrary) bases $\beta_i^+, \beta_i^-$ for cohomology as above, is $Q(\omega) = P(\omega) \cdot P(-\omega)^\top$.
We also point out the following shift relations:
\begin{align} \label{eq:shif_trelations}
\langle f_i^{-1}g^+, f_ig^- \rangle_{\rm ch}^{\omega(s,\nu)} \, = \,  
\langle g^+, g^- \rangle_{\rm ch}^{\omega(s+{ e}_i,\nu)},\ \ \ 
\langle x_jg^+, x_j^{-1}g^- \rangle_{\rm ch}^{\omega(s,\nu)} \, = \,  
\langle g^+, g^- \rangle_{\rm ch}^{\omega(s,\nu+{ e}_j)},
\end{align}
which will be useful later. Here $\omega(s,\nu)$ emphasizes the dependence of $\omega$ on $s,\nu$.

\subsection{Intersection pairing over $K$}
While twisted homology is only defined over $\C$, previous sections have used purely algebraic descriptions of $H^n(X_A,\omega)$ over different rings $A$. This subsection discusses the cohomology intersection pairing $\langle \cdot, \cdot \rangle_{\rm ch}^{\omega}$ over the field $K = \C(s, \nu)$. Our first result says that the function $(s,\nu) \mapsto \langle g^+, g^- \rangle_{\rm ch}^{\omega(s,\nu)}$ belongs to $K$. 
\begin{proposition} \label{prop:rational}
For $g^{\pm} \in H^n(X,\pm \omega)$, the function $(s,\nu) \mapsto \langle g^+,g^- \rangle_{\rm ch}^{\omega(s,\nu)}$ is rational.
\end{proposition}
\begin{proof}
Recall that Serre's duality pairing is a composition
\begin{equation}\label{eq:SerreDuality}
{H}^0(\bar{X},\Omega^n_{\log}(kD))\times {H}^n(\bar{X},\mathcal{O}_{\bar{X}}(-(k+1)D)) \, \overset{\cup}{\longrightarrow} \,  H^n(\bar{X},\Omega^n_{\bar{X}}) \, \overset{\rm tr}{\longrightarrow} \, \C    
\end{equation}
of the {\it cup product} $\cup$ and the {\it trace map} $\rm tr$ \cite[Chapter 3, \S 7]{Hartshorne}.
By the degeneration of the spectral sequence \eqref{eq:specseq} in Lemma \ref{lem:specseq}, we obtain the following representations of $H^n(X,\pm \omega)$:
\begin{align}\label{eq:representations}
\begin{split}
    H^n(X,\omega)&=\frac{{H}^0(\bar{X},\Omega^n_{\log}(kD))}{{\rm im}\left(\nabla_\omega:{H}^0(\bar{X},\Omega^{n-1}_{\log}(kD))\to {H}^0(\bar{X},\Omega^n_{\log}(kD))\right)},    \\
    H^n(X,-\omega)&={\rm ker}\left(\nabla_{-\omega}:{H}^n(\bar{X},\mathcal{O}_{\bar{X}}(-(k+1)D))\to {H}^n(\bar{X},\Omega^1_{\log}(-(k+1)D))\right).    
    \end{split}
\end{align}
Via \eqref{eq:representations}, Serre duality \eqref{eq:SerreDuality} induces a bilinear pairing
$H^n(X,\omega)\times H^n(X,-\omega)\to \C$
, which is identical to $\langle \cdot, \cdot \rangle_{\rm ch}^\omega$ \cite[Eq.~(2.3)]{matsubara2022localization}.
This construction works over $K = \C(s,\nu)$, as Serre duality holds for any projective scheme defined over a field.
The $K$-valued pairing
\begin{equation}\label{eq:intersectionpairing_over_K}
\langle\cdot,\cdot\rangle_{{\rm ch},K}^{\omega}:H^n(X_K,\omega)\times H^n(X_K,-\omega)\to K    
\end{equation}
obtained from Serre duality specializes to $\langle \cdot, \cdot \rangle_{\rm ch}^\omega$ for $s, \nu$ generic as in Definition \ref{def:generic}.
\end{proof}

The cohomology intersection pairing $\langle\cdot,\cdot\rangle_{{\rm ch},K}^{\omega}$ over $K$ is distinguished from other perfect pairings by its compatibility with the $R$-action on twisted cohomology groups, where $R$ is the ring of difference operators from \eqref{eq:commutators}.
The rest of this subsection makes this statement precise. We define another action $\bullet_-$ of $R$ on $\mathcal{O}(X_K)$, slightly different from \eqref{eq:Raction}:
\begin{equation} \label{eq:Raction_minus}
  \s_{s_i} \bullet_- g(s,\nu) =f_i \cdot g(s+e_i,\nu) ,\quad \s_{\nu_j} \bullet_- g(s,\nu) = x_j^{-1} \cdot g(s,\nu+e_j).
\end{equation}
As in the discussion around \eqref{eq:Raction}, the action \eqref{eq:Raction_minus} induces an $R$-action on $H^n(X_K,-\omega)$.
The ring $R$ also acts on $K$ via the shifts $\s_{s_i} \bullet \, a(s,\nu) = a(s+e_i,\nu) $ and $\s_{\nu_j} \bullet  \, a(s,\nu) = a(s,\nu+e_j).$
A $K$-bilinear pairing $\langle\cdot,\cdot\rangle:H^n(X_K,\omega)\times H^n(X_K,-\omega)\to K$ is \emph{compatible} with $R$ if 
    \begin{equation}
    \langle\s \bullet [g_+],\s \bullet_- [g_-]\rangle=\s\bullet \, \langle [g_+],[g_-]\rangle
        \label{eq:compatibility}
    \end{equation}
    holds for any $[g_\pm]\in H^n(X_K,\pm\omega)$ and $\s=\s_{s_i},\s_{\nu_j}$.
Here is a $K$-version of \eqref{eq:shif_trelations}.
\begin{proposition}
The cohomology intersection pairing \eqref{eq:intersectionpairing_over_K} is compatible with $R$.
\end{proposition}

\begin{proof}
Let us write $[\xi_\pm(s,\nu)]\in H^n(X_K,\pm\omega)$ for the cohomology classes in the right-hand side of \eqref{eq:representations}.
Then, $\s_{s_i}\bullet [\xi_+(s,\nu)]$ (resp. $\s_{s_i}\bullet_- [\xi_-(s,\nu)]$) is represented by a cohomology class $[f_i^{-1}\xi_+(s+{ e}_i,\nu)]\in {H}^0(\bar{X}_K,\Omega^n_{\log}(kD-{\rm div}f_i))$ (resp. $[f_i\xi_-(s+{ e}_i,\nu)]\in {H}^n(\bar{X}_K,\mathcal{O}_{\bar{X}_K}(-(k+1)D+{\rm div}f_i))$).
Here, ${\rm div}f_i$ denotes the divisor of $f_i$ viewed as a rational function on $\bar{X}_K$.
We obtain a sequence of identities
\begin{align*}
    \langle \s_{s_i}\bullet [\xi_+(s,\nu)],\s_{s_i}\bullet_- [\xi_-(s,\nu)]\rangle_{{\rm ch},K}^\omega \, 
    &= \, {\rm tr}([f_i^{-1}\xi_+(s+{ e}_i,\nu)]\cup[f_i\xi_-(s+{ e}_i,\nu)]) \\
    &= \, {\rm tr}([\xi_+(s+{e}_i,\nu)]\cup[\xi_-(s+{ e}_i,\nu)])\\
    &= \, \s_{s_i}\bullet {\rm tr}([\xi_+(s,\nu)]\cup[\xi_-(s,\nu)]) \\
    &= \, \s_{s_i}\bullet\langle [\xi_+(s,\nu)],[\xi_-(s,\nu)]\rangle_{{\rm ch},K}^\omega. \qedhere
\end{align*}
\end{proof}

\begin{theorem}\label{thm:uniqueness}
    Up to a non-zero scalar multiplication by $\C$, $\langle \cdot, \cdot \rangle_{{\rm ch},K}^\omega$ is the unique perfect $K$-bilinear pairing $\langle \cdot,\cdot\rangle:H^n(X,\omega)\times H^n(X,-\omega)\to K$ compatible with $R$.
\end{theorem}
Our proof of Theorem \ref{thm:uniqueness} uses the following Lemma. 
\begin{lemma}\label{lem:Schur}
    Let $N$ be a left $R$-module that is finite dimensional over $K$.
    If $N$ is a simple $R$-module, then the dimension of ${\rm End}_R(N)$ over $\C$ is $1$.
\end{lemma}

\begin{proof}
    Let $\varphi\in{\rm End}_R(N)$.
    Writing $\bar{K}$ for the algebraic closure of $K$, the action of $R$ on $K$ extends to that on $\bar{K}$.
    Thus, we may regard $\varphi$ as an element of ${\rm End}_{\bar{K}\otimes_KR}(\bar{N})$ where we set $\bar{N}:=\bar{K}\otimes_KN$.
    We first prove that any eigenvalue of $\varphi$ is in $\C$.
    Let us take an eigenvector $v\in \overline{N}$ of $\varphi$.
    It is straightforward to see that $\s^i_{s_1}v$ is an eigenvector with eigenvalue $\s_{s_1}^i\alpha$.
    Therefore, there exists an integer $i$ so that $\s^i_{s_1}\alpha=\alpha$.
    Similarly, we can prove that $\alpha$ is periodic for all $s_1,\dots,s_\ell,\nu_1,\dots,\nu_n$.
    Since such a function $\alpha\in\bar{K}$ must be a constant function, it must belong to $\C$.
Now, suppose that $\dim_{\C}{\rm End}_{R}(N)\geq 2$ and take $\varphi\in {\rm End}_{R}(N)$ linearly independent from ${\rm id}_N$ over $\C$.
For any eigenvalue $\alpha\in\C$ of $\varphi$, $\alpha \cdot {\rm id}_N-\varphi$ has a non-trivial kernel, which is a non-trivial $R$-submodule of $N$. This is a contradiction.
\end{proof}

\begin{proof}[Proof of Theorem \ref{thm:uniqueness}]
    By Kashiwara's equivalence \cite[Chapter VI, Theorem 7.13]{Borel}, the local cohomology group $M$ defined by \eqref{eq:localcohom} is a simple $D_{\ell+n,\C}$-module.
    It follows from \cite[Theorem 1.2.1]{loeser1991equations} that $H^n(X_K,\omega)$ is a simple $R$-module.
    Any pair of $K$-bilinear pairings $\langle \cdot,\cdot\rangle, \langle \cdot,\cdot\rangle^\prime:H^n(X_K,\omega)\times H^n(X_K,-\omega)\to K$ compatible with $R$
    gives rise to an $R$-morphism $H^n(X_K,\omega)\to H^n(X_K,\omega)$.
    Now, the theorem follows from Lemma \ref{lem:Schur}.
    \end{proof}

\subsection{Degeneration of pairings} \label{subsec:degenpairing}
We now turn to perfect pairings for the likelihood quotient ${\cal O}(X)/I$.
The dual vector space $({\cal O}(X)/I)^\vee$ consists of all linear functionals on ${\cal O}(X)$ which vanish on the {likelihood} ideal $I$.
The \emph{evaluation pairing} $\langle \cdot, \cdot \rangle_{\rm ev}: {\cal O}(X)/I \times  ({\cal O}(X)/I)^\vee \rightarrow \C$ is given by
\[ \langle g, v \rangle_{\rm ev} \, = \, v(g).  \]
Here $g$ is short for $[g]_I$, and $v \in ({\cal O}(X)/I)^\vee$ is such that $v(I) = 0$.
A canonical basis of $({\cal O}(X)/I)^\vee$ is $v_1, \ldots, v_\chi$, where $v_j(g) = g(x^{(j)})/\sqrt{\eta_j}$ represents \emph{evaluation at the $j$-th critical point}.
This is normalized by $\sqrt{\eta_j} = \sqrt{{\rm Hess}(x^{(j)})}$. Together with a basis $[\beta_1]_I, \ldots, [\beta_\chi]_I$ of the likelihood quotient, the evaluation pairing has a matrix representation $E_{ij} = \langle \beta_i, v_j \rangle_{\rm ev} =  \beta_i(x^{(j)})/\sqrt{\eta_j}$. Since the evaluation pairing is perfect, this matrix is invertible. 

In analogy with $Q(\omega) = P(\omega) \cdot P(-\omega)^\top$, we can also consider the bilinear map represented by the matrix $G = E \cdot E^\top$. This is the \emph{Grothendieck residue pairing} $\langle \cdot, \cdot \rangle_{\rm res}$, given by 
\[ {\cal O}(X)/I \times {\cal O}(X)/I \rightarrow \C, \quad \text{with} \quad  \langle g, h \rangle_{\rm res} \, = \, \sum_{j=1}^\chi \, \langle g, v_j \rangle_{\rm ev}  \cdot \langle h, v_j \rangle_{\rm ev} \, = \, \sum_{j=1}^\chi \, \frac{g(x^{(j)})h(x^{(j)})}{\eta_j}.\]
We are now ready to bring in our deformation parameter $\d$. In Section \ref{sec:3}, we did this by replacing $\nabla_\omega$ with $\nabla_\omega^\d$. As we have seen in the proof of Theorem \ref{thm:2intro}, this is equivalent to replacing $\omega$ with $\omega/\d$. Here is what this looks like for our period pairings:
\[ 
\langle g^\pm, \Gamma^\mp \rangle_{\rm per}^{\pm \omega/\d} \, = \, \int_{\Gamma^\mp} g^\pm(x) \cdot L(x)^{\pm \frac{1}{\d}}\, \frac{{\rm d} x_1}{x_1} \wedge \cdots \wedge  \frac{{\rm d} x_n}{x_n}.
\]
We view this as a function of $\d$. If \mychange{$\Gamma_j^{\mp}$} is a Lefschetz thimble, $\log L(x)^{\pm \frac{1}{\d}}$ has constant imaginary part, and its real part reaches a maximum at $x^{(j)}$, see \eqref{eq:decrease}. When $\d \rightarrow 0$, this maximum value at $x^{(j)}$ grows, and the contribution of the rest of the integration contour is more and more suppressed. This is the intuition behind Proposition \ref{prop:degenpairing}, which roughly says that for $\d \rightarrow 0$, \emph{integration} over the Lefschetz thimble turns into \emph{evaluation} at $x^{(j)}$.
\begin{proposition} \label{prop:degenpairing}
Let $\Gamma_j^{\mp}$ be the Lefschetz thimbles associated to the $j$-th critical point $x^{(j)}$ of $\log L(x)$. Under Assumption \ref{assum:lefschetz}, we have the following formulae as $\d \rightarrow 0$:
\begin{align} \label{eq:asymper} \langle g^+, \Gamma_j^- \rangle_{\rm per}^{ \omega/\d} \, &= \, (-2\pi \d)^{\frac{n}{2}} \cdot e^{\frac{1}{\d} \log L(x^{(j)})} \cdot \langle g^{+}, v_j \rangle_{\rm ev} \cdot (1 + O(\d)), \\
 \label{eq:asymper_minus} \langle g^-, \Gamma_j^+ \rangle_{\rm per}^{- \omega/\d} \, &= \, (-2\pi \d)^{\frac{n}{2}}(\sqrt{-1})^{-n} \cdot e^{- \frac{1}{\d} \log L(x^{(j)})} \cdot \langle g^{-}, v_j \rangle_{\rm ev} \cdot (1 + O(\d)).
\end{align}
Similarly, for the cohomology intersection pairing, we have 
\begin{equation} \label{eq:asymint} \langle g^+, g^- \rangle_{\rm ch}^{\omega/\d} \, = \, (2\pi\sqrt{-1} \d)^n \cdot \langle g^+, g^- \rangle_{\rm res} \cdot (1 +  O(\d)). 
\end{equation}
\end{proposition}
\begin{proof}
The formulae \eqref{eq:asymper}-\eqref{eq:asymper_minus} follow from stationary phase approximation \cite[Chapter I]{guillemin1990geometric}. Equation \eqref{eq:asymint} follows from \eqref{eq:intersectionpairing} and \eqref{eq:asymper}-\eqref{eq:asymper_minus}. It appears in \cite[Theorem 2.4]{matsubara2022localization}.
\end{proof}



Propositions \ref{prop:rational} and \ref{prop:degenpairing} lead to a proof of Theorem \ref{thm:1intro}:
\begin{proof}[Proof of Theorem \ref{thm:1intro}]
Let $[\beta_1]_I, \ldots, [\beta_\chi]_I$ be a basis for the likelihood quotient ${\cal O}(X)/I$, and let $P(\pm \omega/\d)_{ij} = \langle \beta_i, \Gamma_j^{\mp} \rangle_{\rm per}^{\pm \omega/\d}$ be the period pairing matrices. Proposition \ref{prop:degenpairing} implies
\[ Q(\omega/\delta) \, = \, P(\omega/\d) \cdot P(\omega/\d)^\top \, = \,  (2 \pi\sqrt{-1} \d)^n \cdot G \cdot (1 + O(\d)).\]
Since the $[\beta_i]_I$ are a basis, the matrix $G$ is invertible, and hence also $Q(\omega/\d)$ is invertible for $\d \rightarrow 0$. Since the entries of $Q(\omega/\d)$ are rational functions of $\d$, see Proposition \ref{prop:rational}, this implies that the classes of the $\beta_i$ form a basis for $H^n(X,\pm \omega/\d)$, for almost all $\d \in \C$. 
\end{proof}
Preferably, we would like to use $\d = 1$ in Theorem \ref{thm:1intro}. Unfortunately, for this, genericity of $s,\nu$ in the sense of Assumption \ref{assum:lefschetz} is not enough. Here is an example.
\begin{example} \label{ex:special}
Let $n=1$ and $f(x)=1-x$. The Euler characteristic of $X$ is $-1$. Genericity in Definition \ref{def:generic} means $\nu,-s,s-\nu\notin\Z$. These three linear forms correspond to the three boundary points $\{0\}, \{1\}, \{\infty\}$ in the compactification $X \subset \mathbb{P}^1$. 
In $H^1(X,\omega)$, we have
\begin{equation} \label{eq:badex}
    \left [ \frac{x+1}{x^2} \right]_{V(\omega)} =\left[ \left( \frac{2\nu-s-1}{\nu-1}\right)\right]_{V(\omega)}.
\end{equation}
If $\nu=1/4$, $s=-1/2$, \eqref{eq:badex} implies that this is not a basis for $H^1(X,\omega)$. However, $[\frac{x+1}{x^2}]_I$ is a basis of ${\cal O}(X)/I$. 
Still, $[\frac{x+1}{x^2}]_{V(\omega/\d)}\in H^1(X,\omega/\d)$ is a basis for generic $\d$.
\end{example}
\begin{remark} \label{rem:delta1}
    We can set $\d = 1$ if we make a stronger genericity assumption on $s, \nu$. In addition to Assumption \ref{assum:lefschetz}, we assume that $\det \, Q(\omega(s,\nu))$ is neither 0 nor $\infty$ at $s, \nu$. Here $Q(\omega(s,\nu)) \in K^{\chi \times \chi}$ is the matrix of rational functions in $s, \nu$ which represents the cohomology intersection pairing for the functions $\beta_i^\pm = \beta_i$ from Theorem \ref{thm:1intro}, see Proposition \ref{prop:rational}. Its determinant is a nonzero rational function, because $Q(\omega(s/\d, \nu/\d))$ is given by~\eqref{eq:asymint}.
\end{remark}

\section{Bases for cohomology and contiguity matrices} \label{sec:5}

This section is about computation. First, we show how to compute a basis for $H^n(X_K,\omega)$. This results in Algorithm \ref{alg:basis}. Second, we recall the definition of \emph{contiguity matrices} and explain how they behave in our degeneration (Theorem \ref{thm:multiplication}). Third, we describe an algorithm for computing contiguity matrices (Algorithm \ref{alg:better_alg}).
This algorithm is related to \emph{Laporta's algorithm} from particle physics \cite{laporta2000high}, which is tailor-made for Feynman integrals. However, in \cite{laporta2000high} there is no mention of contiguity matrices, and the algorithm is more dependent on choices. We believe Algorithm \ref{alg:better_alg} will provide a more systematic way of dealing with families of Feynman integrals, and the generalized Euler integrals from \cite{agostini2022vector}.

\subsection{Computing bases for cohomology}
By Theorem \ref{thm:2intro}, it suffices to compute a subset of ${\cal O}(X_K)$ which represents a constant basis of ${\cal O}(X_K)/I_K$ in the sense of Definition \ref{def:constant}.
Our strategy relies on numerical computation. It is based on some heuristics. However, in practice, it is highly reliable and effective. 
We start by plugging in generic complex values of $s$ and $\nu$ in the likelihood function $L(x)$ from \eqref{eq:likelihood}. We then solve $\omega = {\rm dlog}L(x) = 0$ numerically, using the homotopy continuation technique explained in \cite[Section 5]{agostini2022vector}. This reliably computes all $\chi = (-1)^n \cdot \chi(X)$ complex critical points, even for large Euler characteristics. See \cite{sturmfels2021likelihood} for an example with $\chi = 3628800$. A list of regular functions $\beta_1, \ldots, \beta_\chi \in {\cal O}(X)$ gives a basis of ${\cal O}(X)/I$ if and only if the evaluation pairing from Section \ref{subsec:degenpairing} gives an invertible $\chi \times \chi$-matrix with entries 
$E_{ij}  =  \langle \beta_i, v_j \rangle_{\rm ev}$.

Algorithm \ref{alg:basis} exploits this observation. It takes the likelihood equations as input, as well as a list ${\cal G} 
\subset {\cal O}(X)$ of regular functions. 
The output is a subset of ${\cal G}$ that is maximal independent in ${\cal O}(X_K)/I_K$, i.e.~it has the largest possible cardinality such that its elements are $K$-linearly independent mod $I_K$. One strategy to generate ${\cal G}$ is as follows. 
For fixed $d \geq 1$, we set ${\cal G}_d = \{ f^{-a}x^b \}_{|a| + |b| \leq d}$, with $|a| = a_1 + \cdots + a_\ell$, $a_i \geq 0$, and similarly for $b$.
If the list returned by Algorithm \ref{alg:basis} for ${\cal G} = {\cal G}_d$ contains $m< \chi$ elements, then ${\cal G}_d$ does not contain a basis. 
In that case, we increase $d$ and repeat.
We note that $\mathcal{O}(X_K)/I_K$ is spanned by the union $\bigcup_{d=0}^\infty\mathcal{G}_d$.
Both the computation of the critical points and the rank tests in this algorithm are numerical, but this works well in practice. 
\begin{algorithm}[h!]
\caption{Compute a maximal independent subset of constant functions mod $I_K$}\label{alg:basis}
\begin{algorithmic}
\State \textbf{Input}: $\omega(s,\nu) = {\rm dlog}L(x),\,  {\cal G} = \{ g_1, \ldots, g_N \} \subset {\cal O}(X)$, $g_1 \notin I_K$ 
\State \textbf{Output}: $\{\beta_1, \ldots, \beta_m \} \subset {\cal G}$ such that $[\beta_1]_{I_K}, \ldots, [\beta_m]_{I_K}$ is maximal independent
\State $\omega \gets \omega(s^*,\nu^*)$ for generic complex $s^*, \nu^*$
\State $\{x^{(1)}, \ldots, x^{(\chi)} \} \gets$ solutions of $\omega(x) = 0$
\State $E \gets \text{the row vector } (g_1(x^{(j)}))_{j = 1, \ldots, \chi}$
\State $\beta_1 \gets g_1, \ell \gets 2, k \gets 2$
\While{${\rm rank}(E) < \chi$ and $k \leq N$}
\State $E' \gets$ append the row $(g_k(x^{(j)}))_{j = 1, \ldots, \chi}$ to $E$
\If{${\rm rank} \, E' > {\rm rank} \, E$}
    \State $E \gets E'$, $\beta_\ell \gets g_k$, $\ell \gets \ell + 1$ 
\EndIf
\State $k \gets k + 1$
\EndWhile
\State \textbf{return} $\{\beta_1, \ldots, \beta_{\ell-1} \}$
\end{algorithmic}
\end{algorithm}

\subsection{Contiguity matrices}
Next, our goal is to compute \emph{contiguity matrices} for a given basis $[\beta_1]_{V_K}, \ldots, [\beta_\chi]_{V_K}$. These are $\chi \times \chi$ matrices with entries in $K$ encoding how the difference operators $\sigma_{s_i}, \sigma_{\nu_j} \in R$ act on the basis elements. For instance, the contiguity matrix $C_{s_1}$ satisfies
\[
\s_{s_1} \bullet 
\begin{pmatrix}
    [\beta_1]_{V_K}\\
    \vdots\\
    [\beta_\chi]_{V_K}
\end{pmatrix}
\, =\, \begin{pmatrix}
    \sigma_{s_1} \bullet [\beta_1]_{V_K}\\
    \vdots\\
    \sigma_{s_1} \bullet [\beta_\chi]_{V_K}
\end{pmatrix} \, = \,  C_{s_1}(s,\nu) \cdot 
\begin{pmatrix}
    [\beta_1]_{V_K}\\
    \vdots\\
    [\beta_\chi]_{V_K}
\end{pmatrix}.
\]
Notice that, although the difference operators $\sigma_{s_i}, \sigma_{\nu_j}$ are pairwise commuting, the contiguity matrices are not. This is easily seen in an example: 
\[ \sigma_{s_i} \sigma_{\nu_j} \bullet [\beta]_{V_K} =  \sigma_{s_i} \bullet  C_{\nu_j}(s,\nu)  \cdot [\beta]_{V_K} =   C_{\nu_j}(s+e_i,\nu) \cdot (\sigma_{s_i} \bullet  [\beta]_{V_K} ) =   C_{\nu_j}(s+e_i,\nu) \cdot C_{s_i}(s,\nu) \cdot  [\beta]_{V_K}.\]
Here the second equality applies \eqref{eq:commutators}. Expanding this in the opposite order shows that 
\begin{equation} \label{eq:almostcommute}
C_{\nu_j}(s+e_i,\nu) \cdot C_{s_i}(s,\nu) \, = \, C_{s_i}(s,\nu+e_j) \cdot C_{\nu_j}(s,\nu). 
\end{equation}
More generally, for $a,b \geq 0$, contiguity matrices can be used to compute 
$\sigma_s^a \sigma_\nu^b \bullet [\beta]_{V_K}$ via
\begin{equation} \label{eq:expandab}
 \sigma_s^a \sigma_\nu^b \bullet [\beta]_{V_K} \, = \, {\cal C}_{\nu_n} \cdot \cdots \cdot {\cal C}_{\nu_1} \cdot {\cal C}_{s_\ell} \cdot \cdots \cdot {\cal C}_{s_1} \cdot [\beta]_{V_K},   
\end{equation}
where the calligraphic ${\cal C}$'s denote the following ordered products of matrices: 
\normalsize
\[ {\cal C}_{\nu_j} \, = \, \prod_{q = 1}^{b_j} C_{\nu_j} \big (s+a,\nu + \sum_{k < j} b_k \cdot e_k + (b_j-q) \cdot e_j \big ), \quad {\cal C}_{s_i} \, = \, \prod_{q=1}^{a_i} C_{s_i} \big (s + \sum_{k < i} a_k \cdot e_k + (a_i-q) \cdot e_i, \nu \big ) .\]
Importantly, here, the order of the factors matters: $\prod_{q=1}^{m}C(q):=C(1)C(2)\cdots C(m)$.
As in \eqref{eq:almostcommute}, there are many different ways to expand this as a product of contiguity matrices with shifts in $s$ and $\nu$. When $a,b$ have negative entries, the formula changes slightly. 

The contiguity matrices help to expand the cohomology class $[f^{-a}x^b]_{V_K}$ in terms of a basis for ${\cal O}(X_K)/V_K$. If the basis elements $\beta_i$ are of the form $\beta_i = f^{-a_i}x^{b_i}$,  which will be the case in our algorithm below, the coefficients $c_i$ in $[f^{-a} x^b]_{V_K} = c_1 \, [\beta_1]_{V_K} + \cdots + c_\chi \cdot [\beta_\chi]_{V_K}$ are read from the $i$-th row of \eqref{eq:expandab} for $\sigma_{s}^{a-a_i} \sigma_\nu^{b-b_i} \bullet [\beta]_{V_K}$. 

Contiguity relations for ideals in rings of difference operators can be computed using non-commutative \mychange{\emph{Gr\"obner bases} \cite{mora1986grobner,mora1994introduction}}. Although interesting and important, this paper does not pursue that direction, for two reasons. First, implementations of non-commutative Gr\"obner bases, like the \texttt{OreAlgebra} package in \texttt{Maple}, are not straightforward to use in difference rings like $R$ in which inverses of the shift operators $\sigma$ are allowed. Second, and most importantly, we are convinced that it is important to exploit the knowledge of a basis $B$ for the practical computation of the contiguity matrices. Like in the commutative case, this will speed up the linear algebra computations. Section \ref{sec:computing} takes first steps in that direction.

We now offer a description of contiguity matrices in terms of the degeneration from Section \ref{sec:3}. By the proof of Theorem \ref{thm:2intro}, the set $\{[\beta_1]_{V_{K[\![\d]\!]}}, \ldots, [\beta_\chi]_{V_{K[\![\d]\!]}}\}$ is a basis for $H^n(X_\d, \omega_\d)$ as a free $K[\![\d]\!]$-module. It is straightforward to define contiguity matrices for the $R_\d$-action on $H^n(X_\d, \omega_\d)$: the matrix $C_{\alpha,\d}$ has entries in $K[\![\d]\!]$ and satisfies $\sigma_\alpha \bullet [\beta]_{V_{K[\![\d]\!]}} = C_{\alpha,\d}(s,\nu,\d) \cdot [\beta]_{V_{K[\![\d]\!]}}$, for $\alpha \in \{s_1, \ldots, s_\ell, \nu_1, \ldots, \nu_n\}$. We set
\[ C_{\alpha,\d}(s,\nu,\d) \, = \, C_{\alpha,\d}(s,\nu)_0 + C_{\alpha,\d}(s,\nu)_1 \cdot \d + C_{\alpha,\d}(s,\nu)_2 \cdot \d^2 + \cdots.  \]
The matrix $C_{\alpha,\delta}(s,\nu,\delta)$ is related to $C_\alpha(s,\nu)$ by the relation
\[
C_{\alpha,\delta}(s,\nu,\delta)=C_{\alpha}(s/\d,\nu/\d).
\]
Let $[g]_{V_{K[\![\d]\!]}} = \sum_{i=1}^\chi c_i \cdot [\beta_i]_{K[\![\d]\!]}$ be the expansion of $g$ in terms of the basis elements, and let $\lim_{\d \rightarrow 0} [g]_{V_{K[\![\d]\!]}} \in {\cal O}(X_K)/I_K$ be the image of $[g]_{V_{K[\![\d]\!]}}$ under the limit map $H^n(X_\d, \omega_\d) \rightarrow \lim_{\d \rightarrow 0} H^n(X_\d,\omega_\d) = {\cal O}(X_K)/I_K$ (see Lemma \ref{lem:limithomol}). One easily checks using \eqref{eq:Ractionhbar} that 
\begin{equation} \label{eq:multiplication2}
\lim_{\d \to 0} \left ( \sigma_{s_i} \bullet [g]_{V_{K[\![\d]\!]}} \right ) \, = \, [f_i^{-1} g]_{I_K} \, = \,  c_1' \cdot [\beta_1]_{I_K} + \cdots +  c_\chi' \cdot [\beta_\chi]_{I_K},
\end{equation}
where the vector of coefficients $c' = (c_1', \ldots, c_\chi')^\top \in K^\chi$ is obtained as $C_{s_i,\d}(s,\nu)_0^\top \cdot (c_1, \ldots, c_\chi)^\top$. A similar relation for $\sigma_{\nu_j}$ leads to the following theorem. 
\begin{theorem} \label{thm:multiplication}
    Let $\{\beta_i\}_{i=1}^\chi$ represent a constant basis for ${\cal O}(X_K)/I_K$. Let $C_{\alpha, \d} \in K[\![\d]\!]^{\chi \times \chi}$ be as above. The $\chi \times \chi$ matrices $C_{s_i,\d}(s,\nu)_0^\top$ and $C_{\nu_j,\d}(s,\nu)_0^\top$ with entries in $K$ represent multiplication with $f_i^{-1}$, resp.~$x_j$, in ${\cal O}(X_K)/I_K$, w.r.t.~this basis. Their eigenvalues in the algebraic closure $\bar{K}$ are the evaluations of $f_i^{-1}$, resp.~$x_j$, at the $\chi$ solutions of $\omega(s,\nu) = 0$. 
\end{theorem}
\begin{proof}
    The claim about multiplication in the likelihood quotient ${\cal O}(X_K)/I_K$ follows from \eqref{eq:multiplication2} and its analog for $\alpha = \nu_j$. The statement about eigenvalues follows from the \emph{eigenvalue theorem} in computational algebraic geometry, see for instance \cite[Theorem 3.1.1]{telen2020thesis}.
\end{proof}

\begin{remark}
By Theorem \ref{thm:multiplication}, the matrices $M_{s_i} = C_{s_i,\d}(s,\nu)_0^\top$ and $M_{\nu_i} = C_{\nu_j,\d}(s,\nu)_0^\top$ are pairwise commuting. They share a set of eigenvectors \cite[Theorem  3.1.1]{telen2020thesis}. For a regular function $g = \sum_{i} c_i f^{-a_i}x^{b_i} \in {\cal O}(X_K)$, the eigenvalues of $M_g = \sum_{i} c_i M_s^{a_i} M_{\nu}^{b_i}$ are the evaluations of $g$ at the solutions of $\omega(x) = 0$. In particular, the trace of $M_{g \cdot h} \cdot M_{\rm Hess}^{-1}$ equals the Grothendieck residue pairing of $[g]_{I_K}$ and $[h]_{I_K}$, see Section \ref{sec:4}.
\end{remark}

\subsection{Computing contiguity matrices} \label{sec:computing}
We now turn to our algorithm for computing contiguity matrices. This uses Theorem \ref{thm:RmodJ}, which says that $[f^{-a} x^b]_{V_K} = c_1 \, [\beta_1]_{V_K} + \cdots + c_\chi \cdot [\beta_\chi]_{V_K}$ is equivalent to 
\begin{equation} \label{eq:expansionmodJ}
\sigma_{s}^a \sigma_\nu^b \, - \,  (c_1 \,  \sigma_s^{a_1} \sigma_\nu^{b_1} + \cdots +  c_\chi \, \sigma_s^{a_\chi} \sigma_\nu^{b_\chi}) \, \,  \in  \, J \subset R .
\end{equation}
Let $B = \{ \sigma_{s}^{a_i} \sigma_\nu^{b_i}, i = 1, \ldots, \chi \}$ be the difference operators corresponding to our cohomology basis $\{ [f^{-a_i}x^{b_i}]_{V_K}, i = 1, \ldots, \chi \}$. Consider a larger, finite set $E$ of difference operators of the form $\sigma_s^a\sigma_\nu^b$, containing $B$. 
It is easy to see that $J \cap {\rm span}_K (E)$ has dimension $|E \setminus B|$: a $K$-basis consists of one element of the form \eqref{eq:expansionmodJ} for each $\sigma_s^a \sigma_\nu^b$ in $E \setminus B$. Since our goal is to compute the contiguity matrices $C_{s_i}, C_{\nu_j}$, we will use a subspace $E \supset B$ containing
\[ E_{s_i} = \sigma_{s_i} \bullet B, \, i = 1, \ldots, \ell  \quad \text{and} \quad  E_{\nu_j} = \sigma_{\nu_j} \bullet B, \,  j = 1, \ldots, n. \]
It is convenient to ensure that the span of $E$ contains the $\ell + n$ generators from \eqref{eqn:2}. Let $E_{\rm gen}$ be the set of all monomials $\sigma_s^a \sigma_\nu^b$ that occur with a nonzero coefficient in \eqref{eqn:2}. We set 
\begin{equation} \label{eq:defE}
E \, = \, B \cup E_{s_1} \cup \cdots \cup E_{s_\ell} \cup E_{\nu_1} \cup \cdots \cup E_{\nu_n} \cup E_{\rm gen}. 
\end{equation}
Our goal is to compute a basis for $J \, \cap \, {\rm span}_K(E)$. We are given the subspace ${\cal V} \subset J \, \cap \, {\rm span}_K(E)$ generated by the $\ell + n$ generators of $J$. Very often, this is a strict inclusion, and we need to find more elements in $J \, \cap \, {\rm span}_K(E)$. To do this, we introduce the \emph{plus operator} $\cdot^+$, which is inspired by the \emph{border basis} literature \cite{mourrain1999new}. For a $K$-subspace $S \subset R$, we define 
\begin{align*} S^+ \, =& \,  S + \sigma_{s_1} \bullet S + \cdots + \sigma_{s_\ell} \bullet S + \sigma_{s_1}^{-1} \bullet S + \cdots + \sigma_{s_\ell}^{-1} \bullet S\\
&+ \sigma_{\nu_1} \bullet S + \cdots + \sigma_{\nu_n} \bullet S + \sigma_{\nu_1}^{-1} \bullet S + \cdots + \sigma_{\nu_n}^{-1} \bullet S, 
\end{align*}
and $S^{[k]}$ is $(\cdots((S^+)^+\cdots)^+$, where the plus operator is applied $k$ times. Clearly, ${\cal V}^{[k]} \subset J$ for any $k$, and $\bigcup_{k=0}^\infty {\cal V}^{[k]} = J$. The ascending chain of subspaces 
\begin{equation} \label{eq:ascendingchain}
{\cal V} \, \subset\,  {\cal V}^+ \cap {\rm span}_K(E)  \, \subset \,  \cdots \, \subset \,  {\cal V}^{[k]} \cap {\rm span}_K(E) \, \subset \, \cdots 
\end{equation}
of $E$ stabilizes at finite $k = k^*$, and ${\cal V}^{[k^*]} \cap {\rm span}_K(E) = J \cap {\rm span}_K(E)$. A first, naive algorithm computes a basis for each vector space in the chain \eqref{eq:ascendingchain}, until it detects that $\dim_K ({\cal V}^{[k]} \cap {\rm span}_K(E)) = |E\setminus B|$, which implies that $k = k^*$. This only involves linear algebra with matrices over $K$, as we now expain. 

We start with some notation. Let $S_1, S_2 \subset R$ be finite subsets, such that the elements of $S_2$ are $K$-linearly independent and $S_1 \subset {\rm span}_K(S_2)$. We define a matrix $M(S_1,S_2) \in K^{|S_1| \times |S_2|}$ whose rows are indexed by $S_1$, and the columns are indexed by $S_2$. The row indexed by $P \in S_1$ is given by the coefficients of the unique expansion of $P$ in terms of $S_2$. That is, the entry in row $P \in S_1$ and column $Q \in S_2$ has the coefficient $c_{Q}$ standing with $Q$ in $P= \sum_{Q \in S_2} c_{Q} \, Q$. The row space of $M(S_1,S_2)$ represents ${\rm span}_K(S_1) \subset {\rm span}_K(S_2)$. Finally, for a subset $S_2' \subset S_2$, $M(S_1,S_2)_{S_2'}$ is the submatrix of columns indexed by $S_2'$. 

Let $E^{[k]}$ be the monomial basis $\{ \sigma_s^a\sigma_\nu^b ~|~ \sigma_s^a\sigma_\nu^b  \in {\rm span}_K(E)^{[k]} \}$ of ${\rm span}_K(E)^{[k]}$ and let $V^{[k]}$ be a set of generators for ${\cal V}^{[k]}$. At the $k$-th step in the chain \eqref{eq:ascendingchain}, we construct the matrix $M(V^{[k]},E^{[k]})$. This represents ${\cal V}^{[k]} \subset {\rm span}_K(E)^{[k]}$. To intersect with ${\rm span}_K(E)$, we compute linear combinations of the rows which annihilate the entries in the columns $E^{[k]} \setminus E$. That is, we compute a cokernel (i.e.~left nullspace) matrix $L_k$ of $M(V^{[k]},E^{[k]})_{E^{[k]} \setminus E}$. We have 
\begin{equation} \label{eq:M_k} L_k \cdot M(V^{[k]},E^{[k]})_{E^{[k]}\setminus E} \, = \, 0 \quad \text{and set} \quad M_k \, = \,  L_k \cdot M(V^{[k]},E^{[k]})_E. 
\end{equation}
The following easy lemma states that $M_k$ represents the $k$-th vector space in \eqref{eq:ascendingchain}.
\begin{lemma} \label{lem:Mk}
    The matrix $M_k$ from \eqref{eq:M_k} is $M(W^{[k]},E)$, where $W^{[k]}$ generates ${\cal V}^{[k]} \cap {\rm span}_K(E)$. 
\end{lemma}
Checking if $k = k^*$ amounts to checking that ${\rm rank}\, M_k = |E \setminus B|$. One then replaces $M_k$ by $|E \setminus B|$ of its rows which are linearly independent, and reads off the contiguity relations \eqref{eq:expansionmodJ} from the rows of $(M_k)_{E \setminus B}^{-1} \cdot M_k$.
The algorithm suggested by this discussion has the advantage that it is easy to explain and implement, but it has the disadvantage that it is not very efficient: the size of the set $E^{[k]}$  increases rapidly with $k$. In the rest of the section, we present an improvement which deals with smaller matrices. This will result in Algorithm \ref{alg:better_alg}. 

Let ${\cal V} \subset {\rm span}_K(E)$ be as above and fix a positive integer $k$. We define a sequence ${\cal V} = {\cal V}_{k,0} \subset {\cal V}_{k,1} \subset {\cal V}_{k,2} \subset \cdots$ of subspaces of ${\rm span}_K(E)$ defined recursively as 
\[ {\cal V}_{k,q} \, = \, {\cal V}_{k,q-1}^{[k]} \, \cap \, {\rm span}_K(E). \]
This chain stabilizes at finite $q = q^*$, and ${\cal V}^{[k]} \cap {\rm span}_K(E) \subset {\cal V}_{k,q^*} \subset J \cap {\rm span}_K(E)$. The first inclusion is usually strict, i.e.~${\cal V}^{[k]} \cap {\rm span}_K(E) \subsetneq {\cal V}_{k,q^*}$. In fact, and most importantly, we often have ${\cal V}_{k,q^*} = J \cap {\rm span}_K(E)$ for $k < k^*$ (recall that $k^*$ is the smallest $k$ such that ${\cal V}^{[k]} \cap {\rm span}_K(E) = J \cap {\rm span}_K(E)$). 
It is computationally much less expensive to increase $q$ than to increase $k$: ${\cal V}_{k,q}$ can be computed using matrices of the form $M( - , E^{[k]})$, for any $q$. Hence, this gives us a way to compute $J \cap {\rm span}_K(E)$ by working with smaller matrices.

We present the details of computing ${\cal V}_{k,q^*}$ for fixed $k$. We do this by computing a matrix $M_{k,q^*} = M(V_{k,q^*},E)$ where $V_{k,q^*}$ is a set of generators of ${\cal V}_{k,q^*}$. This is done recursively, starting from $M(V_{k,0},E)$, where ${\cal V}_{k,0} = {\cal V}$. Having computed $V_{k,q-1}$, we proceed by constructing $M(V_{k,q-1}^{[k]}, E^{[k]})$, where $V_{k,q-1}^{[k]}$ is a set of generators for ${\cal V}_{k,q-1}^{[k]}$ (which is easily computed from $V_{k,q-1}$). Similar to what we did 
in \eqref{eq:M_k}, we intersect with ${\rm span}_K(E)$ by computing the cokernel matrix $L_{k,q}$ of $M(V_{k,q-1}^{[k]}, E^{[k]})_{E^{[k]} \setminus E}$, and then setting 
\[ M_{k,q} \, = \, M(V_{k,q},E) \, =  \, L_{k,q} \cdot M(V_{k,q-1}^{[k]}, E^{[k]})_{E}. \]
The stopping criterion for the iteration is that $q = q^*$ if ${\rm rank} \, M_{k,q+1} = {\rm rank} \, M_{k,q}$. If ${\rm rank} \, M_{k,q^*} = {\rm dim}_K {\cal V}_{k,q^*} < |E \setminus B|$, we increase $k$ and repeat.
This is Algorithm \ref{alg:better_alg}.

\begin{algorithm}
\caption{Compute contiguity matrices with respect to a basis $B$}\label{alg:better_alg}
\begin{algorithmic}
\State \textbf{Input}: $B$, the generators $V$ of ${\cal V}$ from \eqref{eqn:2}
\State \textbf{Output}: the contiguity matrices $C_{s_1}, \ldots, C_{s_\ell}, C_{\nu_1}, \ldots, C_{\nu_n}$
\State $E \gets $ the set of monomials $\sigma_s^a \sigma_\nu^b$ from \eqref{eq:defE}
\State $k \gets 0, r \gets 0$
\While{$r < |E \setminus B|$}
\State $k \gets k + 1, \, q \gets 0, r \gets 0$
\State $M_{k,q} \gets M(V,E) = M(V_{k,0},E)$
\While{ $r< {\rm rank} \, M_{k,q} <|E \setminus B|$}
\State $r \gets {\rm rank} \, M_{k,q}$
\State $q \gets q+1$
\State $L_{k,q} \gets$ cokernel matrix of $M(V_{k,q-1}^{[k]},E^{[k]})_{E^{[k]} \setminus E}$
\State $M_{k,q} \gets L_{k,q} \cdot M(V_{k,q-1}^{[k]},E^{[k]})_E = M(V_{k,q},E)$
\EndWhile
\EndWhile
\State $M_{k,q} \gets $ submatrix of $M_{k,q}$ consisting of ${\rm rank} \, M_{k,q}$ linearly independent rows
\State $M_{k,q} \gets (M_{k,q})_{E \setminus B}^{-1} \cdot M_{k,q}$
\For{$\alpha \in \{s_1, \ldots, s_\ell, \nu_1, \ldots, \nu_n\}$}
\State Construct $C_{\alpha}$ by reading the contiguity relations for $\sigma_\alpha \bullet B$ from the rows of $M_{k,q}$
\EndFor
\State \textbf{return} $C_{s_1}, \ldots, C_{s_\ell}, C_{\nu_1}, \ldots, C_{\nu_n}$
\end{algorithmic}
\end{algorithm}

\begin{example} \label{ex:algcubic}
For the data in Example \ref{exa:cubic}, we have $B = \{ 1, \sigma_\nu, \sigma_\nu^2 \}$ and Equation \eqref{eq:defE} gives $E=\{\s_{\nu}^3,\s_{s}\s_{\nu}^2,\s_s\s_{\nu},\s_s,\s_s\s_{\nu}^3 ,1,\s_{\nu},\s_{\nu}^2\}$. For $k = 1$, we find that $q^* = 2$ and $\mathcal{V}_{1,2}\cap {\rm span}_K(E)=J\cap{\rm span}_K(E)$ is represented by the row span of the following $5\times 8$ matrix $M_{1,2} = M(V_{1,2},E)$:
\renewcommand{\kbldelim}{(}
\renewcommand{\kbrdelim}{)}
\[
  M_{1,2}
  = \kbordermatrix{
 &\s_{\nu}^3&\s_{s}\s_{\nu}^2&\s_s\s_{\nu}&\s_s&\s_s\s_{\nu}^3 &1&\s_{\nu}&\s_{\nu}^2\\
&\frac{3s-\nu-3}{\nu}& 0& 0& 0& 0& 1& 0& 0\\
&0& 3s& 0& 0& 0& 0& 0& \nu-3s+2\\
&0& 0& 1& 0& 0& 0& \frac{\nu-3s+1}{3s}& 0\\
&0& 0& 0& 1& \frac{-\nu+3s}{\nu}& 0& 0& 0\\
&0& 0& 0& 1& 0& \frac{\nu-3s}{3s}& 0& 0
  }.
\]
Note that ${\rm rank} \, M_{1,2} = 5 = |E \setminus B|$. The first row of $M_{1,2}$ reads $\nu^{-1}(3s-\nu-3)\s_{\nu}^3+1\in J$. By inverting the leftmost $5\times5$ submatrix, we compute $C_\nu$ in Example \ref{ex:61}. It equals \eqref{eq:multiplication} transposed with $\d = 1$. For $q = 1$, we have ${\rm rank} \, M_{1,1} = 4$. Here $M_{1,1} = M_1$ from Lemma \ref{lem:Mk}. The matrices used to compute $M_{1,2}$ have $|E^+| = |E^{[1]}| = 20$ columns. The naive algorithm implied by Lemma \ref{lem:Mk} requires to compute $M_{2} = M_{2,1}$, using matrices of size $|E^{[2]}| = 36$.
\end{example}

\section{Computational examples} \label{sec:6}
We have implemented Algorithms \ref{alg:basis} and \ref{alg:better_alg} in \texttt{Julia} (v1.8.3). Our code is available at  \url{https://mathrepo.mis.mpg.de/TwistedCohomology}. The numerical solution of $\omega(x)=0$ in Algorithm \ref{alg:basis} relies on the package \texttt{HomotopyContinuation.jl} (v2.6.4) \cite{breiding2018homotopycontinuation} . The symbolic computations in Algoritm \ref{alg:better_alg} are done using \texttt{Oscar.jl} (v0.10.0) \cite{OSCAR}, and they require that $f_i$ have rational coefficients. We tested our implementation for several low-dimensional very affine varieties $X$. This section describes these varieties, and the results. The output of Algorithm \ref{alg:better_alg} consists of $n + \ell$ matrices of size $\chi \times \chi$. Most often, the size of the rational functions in their entries prohibits us from including this output in the paper. All output is available in the form of \texttt{.txt} files at  \url{https://mathrepo.mis.mpg.de/TwistedCohomology}. We used a 16 GB MacBook Pro with an Intel Core i7 processor working at 2.6 GHz.

\begin{example}[Third roots of unity] \label{ex:61}
    Let $n = 1, \ell = 1$ and let $f = 1 - x^3$ as in Examples \ref{exa:cubic}, \ref{ex:lefschetz} and \ref{ex:algcubic}. We keep using the basis $B = \{ 1, \sigma_\nu, \sigma_\nu^2 \}$. As mentioned in Example \ref{ex:algcubic}, we can work with $k = 1$, for which $q^* = 2$. The contiguity matrices are 
    \[ C_\nu \, = \, \begin{pmatrix}
    0 &1 & 0\\ 0& 0& 1\\ \frac{\nu}{\nu-3s+3} & 0 &  0
    \end{pmatrix}, \quad C_s \, = \, \begin{pmatrix}
   \frac{- \nu +3s}{3s} & 0 &  0\\ 0 & \frac{-\nu +3s - 1}{3s} & 0 \\ 0 &  0  & \frac{- \nu +3s - 2}{3s}
    \end{pmatrix}.
    \]
    This can be computed very fast. The same computation for $f = 1- x^{50}$, using $B = \{1, \sigma_\nu, \sigma_\nu^2, \ldots, \sigma_\nu^{49} \}$ takes about half a minute ($k = 1, q^* = 25 $). The reason for this efficiency is that the rational functions in the contiguity matrices are simple. When adding more terms to $f$, the computation time increases. Optimizing our implementation is left as future~work.
\end{example}

\begin{example}[Five points on the line]
We continue Example \ref{ex:M05}, where $X = {\cal M}_{0,5}$. This space and the associated generalized Euler integrals appear in physics in the context of five point string amplitudes, see \cite[Equation (4.7)]{arkani2021stringy} and \cite[Appendix A]{mizera2018scattering}. In the basis $B = \{ 1, \sigma_{\nu_1} \}$, we need $k = 1$ and $q^* = 2$ to compute the contiguity relations. We find that
\[ C_{\nu_1} \, = \, \begin{pmatrix}
    0 & 1 \\ r_1 & r_2
\end{pmatrix}, \quad C_{\nu_2} \, = \, \begin{pmatrix} \frac{\nu_1+\nu_2-s_3}{\nu_2 - s_2 - s_3 + 1} & \frac{-\nu_1 + s_1 + s_3 - 1}{\nu_2 - s_2 - s_3 + 1} \\ r_3 & r_4,\end{pmatrix}\]
where $r_1,r_2,r_3,r_4$ are the rational functions
\small
\begin{align*}
    r_1 &\,= \,\frac{\nu_1(-\nu_1 - \nu_2  + s_3)}{(\nu_1 - s_1 - s_3 + 2)  (\nu_1 + \nu_2 - s_1 - s_2 - s_3 + 2)}, \\
r_2 &\,= \, \frac{\nu_1(2 \nu_1 + 2 \nu_2 - 2  s_1 - s_2 - 3 s_3 + 4 )- \nu_2 (s_1 +  s_3-2) + s_3(s_1 + s_2  + s_3 - 3) - s_1 - s_2  + 2}{(\nu_1 - s_1 - s_3 + 2)  (\nu_1 + \nu_2 - s_1 - s_2 - s_3 + 2)}, \\
    r_3 &\,= \,\frac{\nu_1(\nu_1 + \nu_2  - s_3)}{(\nu_2 - s_2 - s_3 + 1)  (\nu_1 + \nu_2 - s_1 - s_2 - s_3 + 2)}, \\
 r_4 &\,= \,\frac{\nu_1(-\nu_1 + s_1 + s_3 -1) + \nu_2( \nu_2 - s_2 - s_3 + 1)}{(\nu_2 - s_2 - s_3 + 1)  (\nu_1 + \nu_2 - s_1 - s_2 - s_3 + 2)}.
\end{align*} 
\normalsize
The algorithm also returns the matrices $C_{s_1}, C_{s_2}, C_{s_3}$, whose entries are slightly more complicated. While this runs in less than a second, the same computation for the three-dimensional moduli space ${\cal M}_{0,6}$ (with Euler characteristic $-6$) does not terminate within reasonable time. The parameters are $n = 3, \ell = 6$ and the $f_i$ are the bottom two rows of \cite[Equation (6)]{sturmfels2021likelihood}. This is a nice computational challenge for future improvements of Algorithm \ref{alg:better_alg}. 
\end{example}

\begin{example}[$k \neq 1$]
    We set $n = 2,\, \ell = 1$ and consider $f = 1 + x^2 + y^3 + x^2y^3$. The Euler characteristic is $6$. Algorithm \ref{alg:basis} selects $B = \{ 1, \sigma_{\nu_2}, \sigma_{\nu_2}^2, \sigma_{\nu_1}, \sigma_{\nu_1} \sigma_{\nu_2}, \sigma_{\nu_1} \sigma_{\nu_2}^2 \}$ among all monomials of degree at most three. In this example, for $k = 1$, we have $q^* = 1$ and ${\rm rank} \, M_{1,1} = 4 < |E \setminus B| = 14$. It suffices to increase $k$ by 1: for $k =2$, we find $q^* = 2$ and ${\rm rank} \, M_{2,2} = 14$. The computation takes less than three seconds in total.  
\end{example}

\begin{example}[Fermat hypersurfaces]
This example uses $n = 2, 3, 4$, $\ell = 1$ and $f = x_1^d + \cdots + x_n^d - 1$, for $d \geq 1$. The very affine variety $X$ is the complement of a Fermat hypersurface in the $n$-dimensional torus. We have $\chi = d^n$ and use $B = \{ \sigma_{\nu_1}^{d_1} \cdots \sigma_{\nu_n}^{d_n} \, : \, 0 \leq d_i \leq d-1 \}$. For $n = 2$, we compute contiguity matrices for the Fermat curve of degree $d=10$ within less than five minutes. The matrices are obtained from $M_{k,q^*}$ with $k = 1, q^* = 11$. For surfaces ($n=3$), the computation for $d = 4$ runs in about five minutes, and the result is obtained for $k = 1, q^* = 8$. For $n = 4,k = 1, q^* = 4$, it takes about 30 seconds to compute the five contiguity matrices for the quadratic threefold $x_1^2 + x_2^2 + x_3^2 + x_4^2 -1 = 0$.
\end{example}

\begin{example}[Feynman integrals]
Feynman integrals from physics are of the form \eqref{eq:per}. They are associated to a graph $G$, called \emph{Feynman diagram}, which encodes particle interaction patterns. In this context $\ell = 1$, and the polynomial $f$ in the likelihood function $L = f^s x^\nu$ is the \emph{graph polynomial} associated to $G$. The graph polynomial is the sum of the first and second \emph{Symanzik polynomials} of $G$. The number of variables is the number of \emph{internal edges} $n$. Details are in \cite{BBKP,lee2013critical,mizera2022landau}. Here, we apply our algorithm for two different graphs. Both are examples of \emph{one-loop diagrams} \cite[Section 2.5]{mizera2022landau}. They are shown in Figure \ref{fig:feynman}.  The first one is called the \emph{bubble diagram}. The very affine variety is the complement~of 
\[f_{\rm bubble} \, = \, c_{11} x_1^2 + c_{12} x_1x_2 + c_{22} x_2^2 + x_1 + x_2 \, = \, 0 \]
in $(\mathbb{C}^*)^2$, where the $c_{ij}$ depend on masses and momenta.  Here $n = 2$, and the internal edges are those labeled $m_1$, $m_2$. We arbitrarily chose $(c_{11}, c_{12}, c_{22}) = (7,12,3)$. The Euler characteristic is $\chi(X) = 3$. With basis $B = \{ 1, \sigma_{\nu_1}, \sigma_{\nu_2} \}$, the contiguity matrices are found for $k = 1, q^* = 2$. The next diagram is the \emph{triangle diagram} with massless internal particles: 
\[ f_{\rm triangle} \, = \, c_{12} x_1 x_2 + c_{13} x_1 x_3 + c_{23} x_2 x_3 + x_1 + x_2 + x_3. \]
The very affine variety $X$ is a threefold, i.e.~$n = 3$. We set $(c_{12}, c_{13}, c_{23}) = (2,-6,-8)$. The Euler characteristic is $-4$. Our algorithm computes the $n + \ell = 4$ contiguity matrices of size $4 \times 4$ within less than a second. We used $B = \{ 1, \sigma_{\nu_1}, \sigma_{\nu_2}, \sigma_{\nu_3} \}$, and $k = 1, q^* = 2$. 
\begin{figure}
\centering
\includegraphics[height = 2.5cm]{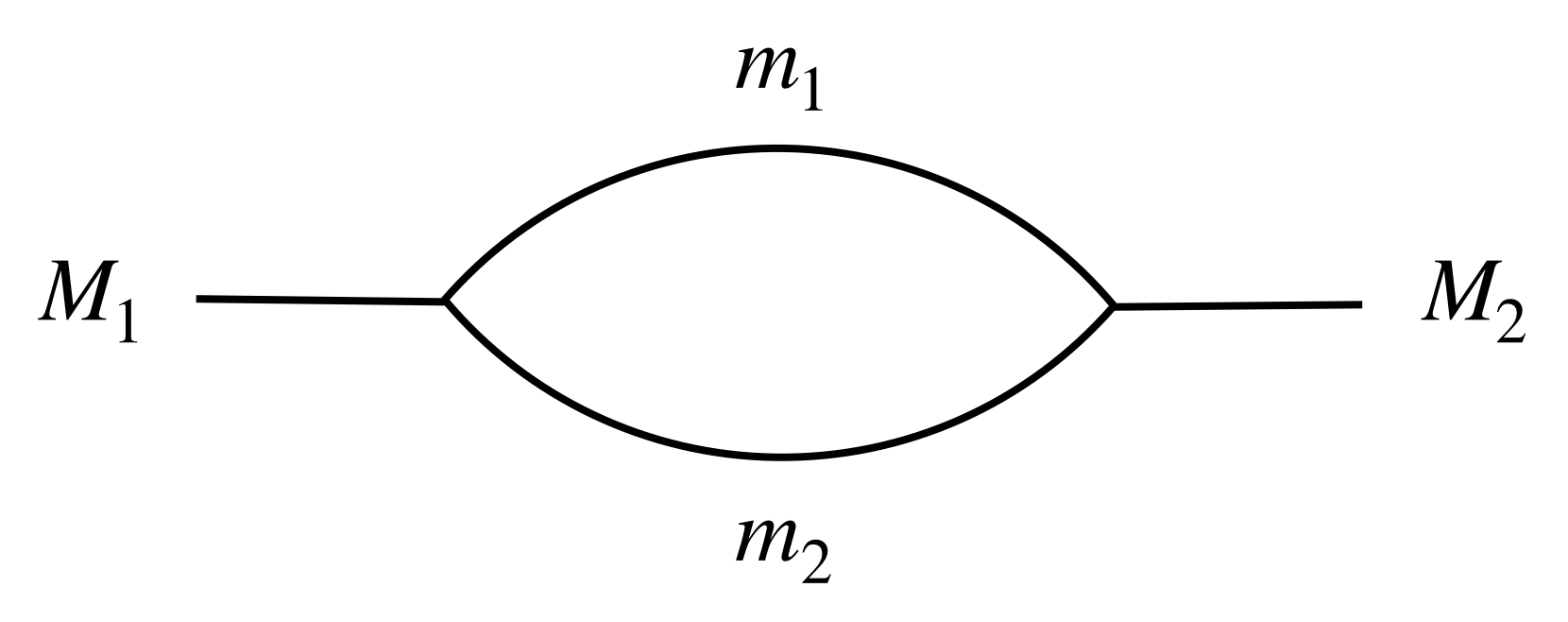} \quad \quad
\includegraphics[height = 3.3cm]{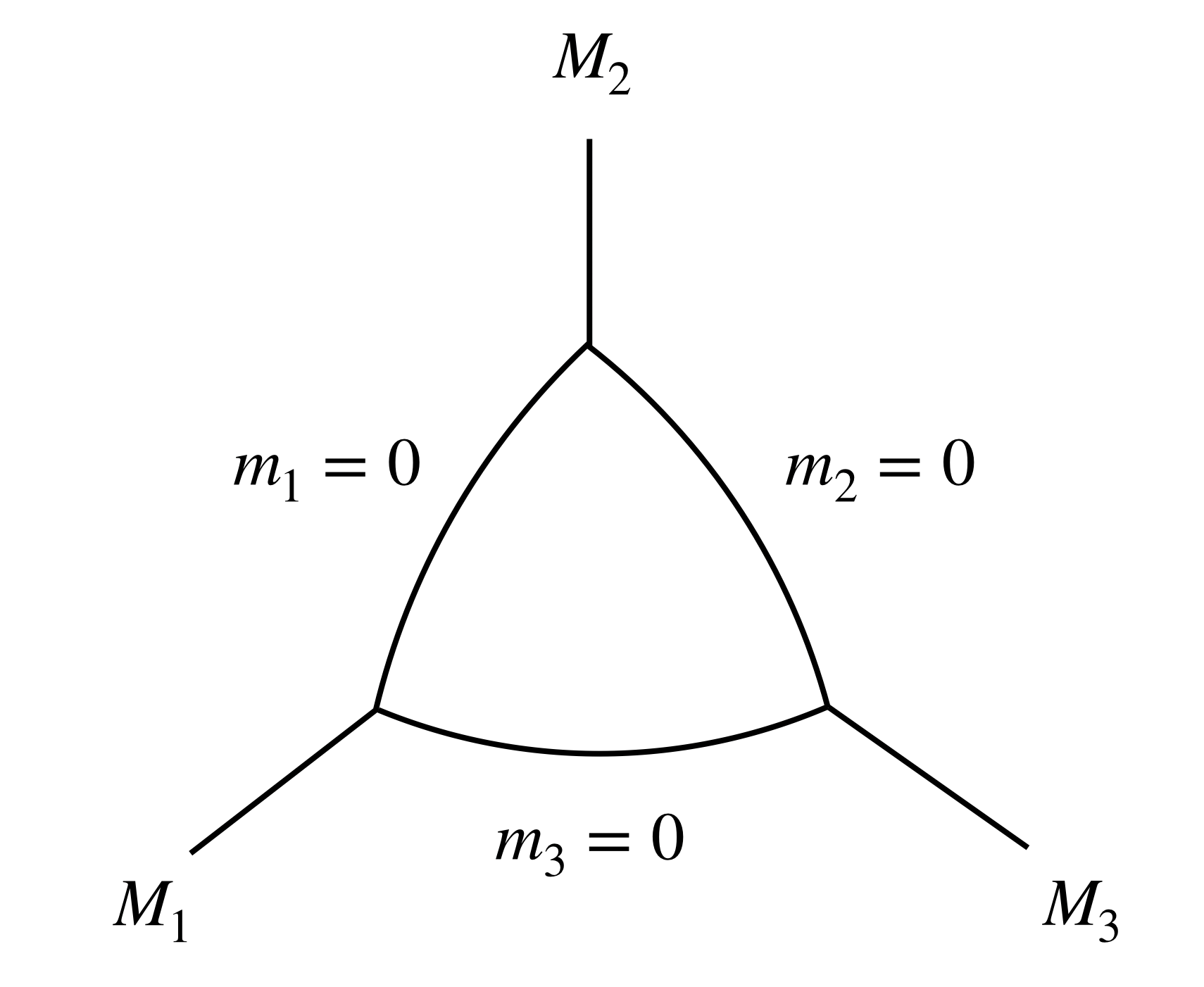}
\caption{Massive bubble diagram and triangle diagram with zero internal masses. }
\label{fig:feynman}
\end{figure}
\end{example}

\section*{Acknowledgements}
Saiei-Jaeyeong Matsubara-Heo was supported
by JSPS KAKENHI Grant Number 19K14554 and 22K13930, and partially supported by JST CREST Grant Number JP19209317. Simon Telen was supported by a Veni grant from the Netherlands Organisation for Scientific Research (NWO).
The first author is grateful to the nonlinear algebra group at MPI MiS Leipzig for its hospitality. Discussions during his visit led to the Appendix of \cite{agostini2022vector} and to this work. We thank Bernd Sturmfels for his helpful comments to an earlier version of this paper. 
\bibliographystyle{abbrv}
\bibliography{references.bib}

\noindent{\bf Authors' addresses:}
\medskip

\noindent Saiei-Jaeyeong Matsubara-Heo, Kumamoto University \hfill{\tt saiei@educ.kumamoto-u.ac.jp}

\noindent Simon Telen, MPI-MiS Leipzig and CWI Amsterdam {\em (current)}
\hfill {\tt simon.telen@mis.mpg.de}

\end{document}